\documentclass[12pt]{article}
\usepackage{amsmath,amssymb}
\usepackage{graphicx}
\newtheorem{theorem}{Theorem}[section]
\newtheorem{proposition}[theorem]{Proposition}
\newtheorem{lemma}[theorem]{Lemma}
\def\bull{\vrule height .9ex width .8ex depth -.1ex}
\newenvironment{proof}{\smallbreak \noindent {\bf Proof.~}}
              {\unskip\nobreak\hfill\hskip 2em \bull\par\medbreak}
\newenvironment{proofof}[1]{\medbreak\noindent{\bf Proof of~#1.~}}
              {\unskip\nobreak\hfill\hskip 2em \bull\par\medbreak}
\def\cA{\mathcal{A}}
\def\cB{\mathcal{B}}
\def\cC{\mathcal{C}}
\def\cG{\mathcal{G}}
\def\cT{\mathcal{T}}
\def\cU{\mathcal{U}}
\def\bN{\mathbb{N}}
\def\bZ{\mathbb{Z}}
\def\fB{\mathfrak{B}}
\def\fU{\mathfrak{U}}
\def\fW{\mathfrak{W}}
\def\al{\alpha}
\def\be{\beta}
\def\ga{\gamma}
\def\Ga{\Gamma}
\def\de{\delta}
\def\De{\Delta}
\def\om{\omega}
\def\Om{\Omega}
\def\eps{\epsilon}
\def\MG{\mathcal{MG}}
\def\oB{\overline{B}}
\def\St{\mathop{\mathrm{St}}}
\def\Sch{\mathop{\mathrm{Sch}}}
\def\adj{\mathop{\mathrm{adj}}}
\def\coset{\mathop{\mathrm{coset}}}
\def\Sub{\mathop{\mathrm{Sub}}}
\def\Fix{\mathop{\mathrm{Fix}}}

\def\dT{\partial\mathcal{T}}
\def\1{1_{\mathcal{G}}}
\title{Notes on the Schreier graphs\\ of the Grigorchuk group}
\author{Yaroslav Vorobets}
\date{}
\begin{document}

\maketitle

\begin{abstract}
The paper is concerned with the space of the marked Schreier graphs of the
Grigorchuk group and the action of the group on this space.  In particular,
we describe the invariant set of the Schreier graphs corresponding to the
action on the boundary of the binary rooted tree and dynamics of the group
action restricted to this invariant set.
\end{abstract}

\section{Introduction}\label{intro}

This paper is devoted to the study of two equivalent dynamical systems of
the Grigorchuk group $\cG$, the action on the space of the marked Schreier
graphs and the action on the space of subgroups.  The main object of study
is going to be the set of the marked Schreier graphs of the standard action
of the group on the boundary of the binary rooted tree and their limit
points in the space of all marked Schreier graphs of $\cG$.

Given a finitely generated group $G$ with a fixed generating set $S$, to
each action of $G$ we associate its Schreier graph, which is a
combinatorial object that encodes some information about orbits of the
action.  The marked Schreier graphs of various actions form a topological
space $\Sch(G,S)$ and there is a natural action of $G$ on this space.  Any
action of $G$ corresponds to an invariant set in $\Sch(G,S)$ and any action
with an invariant measure gives rise to an invariant measure on
$\Sch(G,S)$.  The latter allows to define the notion of a random Schreier
graph, which is closely related to the notion of a random subgroup of $G$.

A principal problem is to determine how much information about the original
action can be learned from the Schreier graphs.  The worst case here is a
free action, for which nothing beyond its freeness can be recovered.
Vershik \cite{V} introduced the notion of a totally nonfree action.  This
is an action such that all points have distinct stabilizers.  In this case
the information about the original action can be recovered almost
completely.  Further extensive development of these ideas was done by
Grigorchuk \cite{G2}.

The Grigorchuk group was introduced in \cite{G0} as a simple example of a
finitely generated infinite torsion group.  Later it was revealed that this
group has intermediate growth and a number of other remarkable properties
(see the survey \cite{G1}).  In this paper we are going to use the
branching property of the Grigorchuk group, which implies that its action
on the boundary of the binary rooted tree is totally nonfree in a very
strong sense.

\begin{figure}[t]
\centerline{\includegraphics{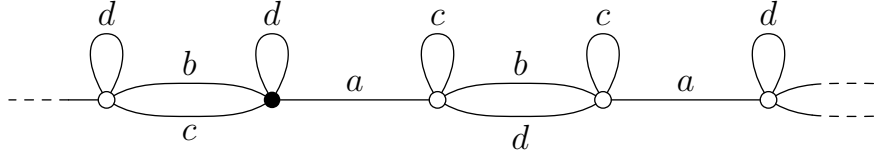}}
\caption{The marked Schreier graph of $0^\infty=000\dots$}
\label{fig1}
\end{figure}

The main results of the paper are summarized in the following two theorems.
The first theorem contains a detailed description of the invariant set of
the Schreier graphs.  The second theorem is concerned with the dynamics of
the group action restricted to that invariant set.

\begin{theorem}\label{main1}
Let $F:\dT\to\Sch(\cG,\{a,b,c,d\})$ be the mapping that assigns to any
point on the boundary of the binary rooted tree the marked Schreier graph
of its orbit under the action of the Grigorchuk group.  Then
\begin{itemize}
\item[(i)]
$F$ is injective;
\item[(ii)]
$F$ is measurable; it is continuous everywhere except for a countable set,
the orbit of the point $\xi_0=111\dots$;
\item[(iii)]
the Schreier graph $F(\xi_0)$ is an isolated point in the closure of
$F(\dT)$; the other isolated points are graphs obtained by varying the
marked vertex of $F(\xi_0)$;
\item[(iv)]
the closure of the set $F(\dT)$ differs from $F(\dT)$ in countably many
points; these are obtained from three graphs $\De_0,\De_1,\De_2$ choosing
the marked vertex arbitrarily;
\item[(v)] as an unmarked graph, $F(\xi_0)$ is a double quotient of each
$\De_i$ ($i=0,1,2$); also, there exists a graph $\De$ such that each
$\De_i$ is a double quotient of $\De$.
\end{itemize}
\end{theorem}

\begin{theorem}\label{main2}
Using notation of the previous theorem, let $\Om$ be the set of
non-isolated points of the closure of $F(\dT)$.  Then
\begin{itemize}
\item[(i)]
$\Om$ is a minimal invariant set for the action of the Grigorchuk group
$\cG$ on $\Sch(\cG,\{a,b,c,d\})$;
\item[(ii)]
the action of $\cG$ on $\Om$ is a continuous extension of the action on the
boundary of the binary rooted tree; the extension is one-to-one everywhere
except for a countable set, where it is three-to-one;
\item[(iii)]
there exists a unique Borel probability measure $\nu$ on
$\Sch(\cG,\{a,b,c,d\})$ invariant under the action of $\cG$ and supported
on the set $\Om$;
\item[(iv)]
the action of $\cG$ on $\Om$ with the invariant measure $\nu$ is isomorphic
to the action of $\cG$ on $\dT$ with the uniform measure.
\end{itemize}
\end{theorem}

The paper is organized as follows.  Section \ref{graph} contains a detailed
construction of the space of marked graphs.  The construction is more
general than that in \cite{G2}.  Section \ref{act} contains notation and
definitions concerning group actions.  In Section \ref{sch} we introduce
the Schreier graphs of a finitely generated group, the space of marked
Schreier graphs, and the action of the group on that space.  In Section
\ref{sub} we study the space of subgroups of a countable group and
establish a relation of this space with the space of marked Schreier
graphs.  Section \ref{tree} is devoted to general considerations concerning
groups of automorphisms of a regular rooted tree and their actions on the
boundary of the tree.  In Section \ref{grig} we apply the results of the
previous sections to the study of the Grigorchuk group and prove Theorems
\ref{main1} and \ref{main2}.  The exposition in Sections
\ref{graph}--\ref{tree} is as general as possible, to make their results
applicable to the actions of groups other than the Grigorchuk group.

The author is grateful to Rostislav Grigorchuk, Anatoly Vershik, and Oleg
Ageev for useful discussions.

\begin{figure}[t]
\centerline{\includegraphics{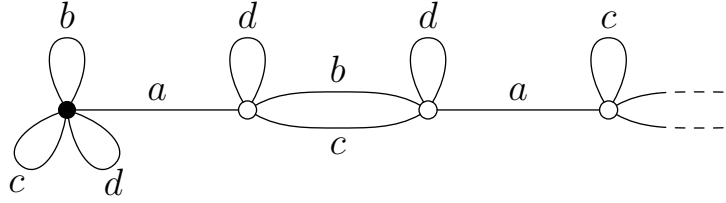}}
\caption{The marked Schreier graph of $1^\infty=111\dots$}
\label{fig2}
\end{figure}

\section{Space of marked graphs}\label{graph}

A {\em graph\/} $\Ga$ is a combinatorial object that consists of {\em
vertices\/} and {\em edges\/} related so that every edge joins two vertices
or a vertex to itself (in the latter case the edge is called a {\em loop}).
The vertices joined by an edge are its {\em endpoints}.  Let $V$ be the
{\em vertex set\/} of the graph $\Ga$ and $E$ be the set of its edges.
Traditionally $E$ is regarded as a subset of $V\times V$, i.e., any edge is
identified with the pair of its endpoints.  In this paper, however, we are
going to consider graphs with multiple edges joining the same vertices.
Also, our graphs will carry additional structure.  To accomodate this, we
regard $E$ merely as a reference set whereas the actual information about
the edges is contained in their {\em attributes\/}, which are functions on
$E$.  In a plain graph any edge has only one attribute: its endpoints,
which are an unordered pair of vertices.  Other types of graphs involve
more attributes.

A {\em directed graph\/} has directed edges.  The endpoints of a {\em
directed edge\/} $e$ are ordered, namely, there is a {\em beginning\/}
$\al(e)\in V$ and an {\em end\/} $\om(e)\in V$.  Clearly, an undirected
loop is no different from a directed one.  An undirected edge joining two
distinct vertices may be regarded as two directed edges $e_1$ and $e_2$
with the same endpoints and opposite directions, i.e., $\al(e_2)=\om(e_1)$
and $\om(e_2)=\al(e_1)$.  This way we can represent any graph with
undirected edges as a directed graph.  Conversely, some directed graphs can
be regarded as graphs with undirected edges (we shall use this in Section
\ref{grig}).

A {\em graph with labeled edges\/} is a graph in which each edge $e$ is
assigned a {\em label\/} $l(e)$.  The labels are elements of a prescribed
finite set.  A {\em marked graph\/} is a graph with a distinguished vertex
called the {\em marked vertex}.

The vertices of a graph are pictured as dots or small circles.  An
undirected edge is pictured as an arc joining its endpoints.  A directed
edge is pictured as an arrow going from its beginning to its end.  The
label of an edge is written next to the edge.  Alternatively, one might
think of labels as colors and picture a graph with labeled edges as a
colored graph.

Let $\Ga$ be a graph and $V$ be its vertex set.  To any subset $V'$ of $V$
we associate a graph $\Ga'$ called a {\em subgraph\/} of $\Ga$.  By
definition, the vertex set of the graph $\Ga'$ is $V'$ and the edges are
those edges of $\Ga$ that have both endpoints in $V'$ (all attributes are
retained).  If $\Ga$ is a marked graph and the marked vertex is in $V'$, it
will also be the marked vertex of the subgraph $\Ga'$.  Otherwise the
subgraph is not marked.

Suppose $\Ga_1$ and $\Ga_2$ are graphs of the same type.  For any
$i\in\{1,2\}$ let $V_i$ be the vertex set of $\Ga_i$ and $E_i$ be the set
of its edges.  The graph $\Ga_1$ is said to be {\em isomorphic\/} to
$\Ga_2$ if there exist bijections $f:V_1\to V_2$ and $\phi:E_1\to E_2$ that
respect the structure of the graphs.  First of all, this means that $f$
sends the endpoints of any edge $e\in E_1$ to the endpoints of $\phi(e)$.
If $\Ga_1$ and $\Ga_2$ are directed graphs, we additionally require that
$\al(\phi(e))=f(\al(e))$ and $\om(\phi(e))=f(\om(e))$ for all $e\in E_1$.
If $\Ga_1$ and $\Ga_2$ have labeled edges, we also require that $\phi$
preserve labels.  If $\Ga_1$ and $\Ga_2$ are marked graphs, we also require
that $f$ map the marked vertex of $\Ga_1$ to the marked vertex of $\Ga_2$.
Assuming the above requirements are met, the mapping $f$ of the vertex set
is called an {\em isomorphism of the graphs\/} $\Ga_1$ and $\Ga_2$.  If
$\Ga_1=\Ga_2$ then $f$ is also called an {\em automorphism\/} of the graph
$\Ga_1$.  We call the mapping $\phi$ a {\em companion mapping\/} of $f$.
If neither of the graphs $\Ga_1$ and $\Ga_2$ admits multiple edges with
identical attributes, the companion mapping is uniquely determined by the
isomorphism $f$.  Further, we say that the graph $\Ga_2$ is a {\em
quotient\/} of $\Ga_1$ if all of the above requirements are met except the
mappings $f$ and $\phi$ need not be injective.  Moreover, $\Ga_2$ is a {\em
$k$-fold quotient\/} of $\Ga_1$ if $f$ is $k$-to-$1$.  Finally, we say that
the graphs $\Ga_1$ and $\Ga_2$ coincide {\em up to renaming edges\/} if
they have the same vertices and there is a one-to-one correspondence
between their edges that preserves all attributes.  An equivalent condition
is that the identity map on the common vertex set is an isomorphism of
these graphs.

A {\em path\/} in a graph $\Ga$ is a sequence of vertices
$v_0,v_1,\dots,v_n$ together with a sequence of edges $e_1,\dots,e_n$ such
that for any $1\le i\le n$ the endpoints of the edge $e_i$ are $v_{i-1}$
and $v_i$.  We say that the vertex $v_0$ is the beginning of the path and
$v_n$ is the end.  The path is {\em closed\/} if $v_n=v_0$.  The {\em
length\/} of the path is the number of edges in the sequence (counted with
repetitions), which is a nonnegative integer.  The path is a {\em directed
path\/} if the edges are directed and, moreover, $\al(e_i)=v_{i-1}$ and
$\om(e_i)=v_i$ for $1\le i\le n$.  If the graph $\Ga$ has labeled edges
then the path is assigned a {\em code word\/} $l(e_1)l(e_2)\dots l(e_n)$,
which is a string of labels read off the edges while traversing the path.

We say that a vertex $v$ of a graph $\Ga$ is {\em connected to\/} a vertex
$v'$ if there is a path in $\Ga$ such that the beginning of the path is $v$
and the end is $v'$.  The length of the shortest path with this property is
the {\em distance\/} from $v$ to $v'$.  The connectivity is an equivalence
relation on the vertex set of $\Ga$.  The subgraphs of $\Ga$ corresponding
to the equivalence classes are {\em connected components\/} of the graph
$\Ga$.  A graph is {\em connected\/} if all vertices are connected to each
other.  Clearly, the connected components of any graph are its maximal
connected subgraphs.

Let $v$ be a vertex of a graph $\Ga$.  For any integer $n\ge0$, the {\em
closed ball\/} of radius $n$ centered at $v$, denoted $\oB_\Ga(v,n)$, is
the subgraph of $\Ga$ whose vertex set consists of all vertices in $\Ga$ at
distance at most $n$ from the vertex $v$.  A graph is {\em locally
finite\/} if every vertex is the endpoint for only finitely many edges.  If
the graph $\Ga$ is locally finite then any closed ball of $\Ga$ is a {\em
finite graph}, i.e., it has a finite number of vertices and a finite number
of edges.

Let $\MG$ denote the set of isomorphism classes of all marked directed
graphs with labeled edges.  For convenience, we regard elements of $\MG$ as
graphs (i.e., we choose representatives of isomorphism classes).  It is
easy to observe that connectedness and local finiteness of graphs are
preserved under isomorphisms.  Let $\MG_0$ denote the subset of $\MG$
consisting of connected, locally finite graphs.  We endow the set $\MG_0$
with a topology as follows.  The topology is generated by sets
$\cU(\Ga_0,V_0)\subset\MG_0$, where $\Ga_0$ runs over all finite graphs in
$\MG_0$ and $V_0$ can be any subset of the vertex set of $\Ga_0$.  By
definition, $\cU(\Ga_0,V_0)$ is the set of all isomorphism classes in
$\MG_0$ containing any graph $\Ga$ such that $\Ga_0$ is a subgraph of $\Ga$
and every edge of $\Ga$ with at least one endpoint in the set $V_0$ is
actually an edge of $\Ga_0$.  In other words, there is no edge in $\Ga$
that joins a vertex from $V_0$ to a vertex outside the vertex set of
$\Ga_0$.  For example, $\cU(\Ga_0,\emptyset)$ is the set of all graphs in
$\MG_0$ that have a subgraph isomorphic to $\Ga_0$.  On the other hand, if
$V_0$ is the entire vertex set of $\Ga_0$ then $\cU(\Ga_0,V_0)$ contains
only the graph $\Ga_0$.  As a consequence, every finite graph in $\MG_0$ is
an isolated point.  The following lemma implies that sets of the form
$\cU(\Ga_0,V_0)$ constitute a base of the topology.

\begin{lemma}\label{graph1}
Any nonempty intersection of two sets of the form $\cU(\Ga_0,V_0)$ can be
represented as the union of some sets of the same form.
\end{lemma}

\begin{proof}
Let $\Ga_1,\Ga_2\in\MG_0$ be finite graphs and $V_1,V_2$ be subsets of
their vertex sets.  Consider an arbitrary graph $\Ga\in\cU(\Ga_1,V_1)\cap
\cU(\Ga_2,V_2)$.  For any $i\in\{1,2\}$ let $f_i:W_i\to W'_i$ be an
isomorphism of the graph $\Ga_i$ with a subgraph of $\Ga$ such that no edge
of $\Ga$ joins a vertex from the set $f_i(V_i)$ to a vertex outside $W'_i$.
Denote by $\Ga_0$ the finite subgraph of $\Ga$ with the vertex set
$W_0=W'_1\cup W'_2$.  Since the subgraphs of $\Ga$ with vertex sets $W'_1$
and $W'_2$ are both connected and both contain the marked vertex of $\Ga$,
the subgraph $\Ga_0$ is also marked and connected.  Besides, no edge of
$\Ga$ joins a vertex from the set $V_0=f_1(V_1)\cup f_2(V_2)$ to a vertex
outside $W_0$.  Hence $\Ga\in\cU(\Ga_0,V_0)$.  It is easy to observe that
the entire set $\cU(\Ga_0,V_0)$ is contained in the intersection
$\cU(\Ga_1,V_1)\cap\cU(\Ga_2,V_2)$.  The lemma follows.
\end{proof}

Next we introduce a distance function on $\MG_0$.  Consider arbitrary
graphs $\Ga_1,\Ga_2\in\MG_0$.  Let $v_1$ be the marked vertex of $\Ga_1$
and $v_2$ be the marked vertex of $\Ga_2$.  We let $\de(\Ga_1,\Ga_2)=0$ if
the graphs $\Ga_1$ and $\Ga_2$ are isomorphic (i.e., they represent the
same element of $\MG_0$).  Otherwise we let $\de(\Ga_1,\Ga_2)=2^{-n}$,
where $n$ is the smallest nonnegative integer such that the closed balls
$\oB_{\Ga_1}(v_1,n)$ and $\oB_{\Ga_2}(v_2,n)$ are not isomorphic.

\begin{lemma}\label{graph2}
The graphs $\Ga_1$ and $\Ga_2$ are isomorphic if and only if the closed
balls $\oB_{\Ga_1}(v_1,n)$ and $\oB_{\Ga_2}(v_2,n)$ are isomorphic for any
integer $n\ge0$.
\end{lemma}

\begin{proof}
For any $i\in\{1,2\}$ let $V_i$ denote the vertex set of the graph $\Ga_i$
and $E_i$ denote its set of edges.  Further, for any integer $n\ge0$ let
$V_i(n)$ and $E_i(n)$ denote the vertex set and the set of edges of the
closed ball $\oB_{\Ga_i}(v_i,n)$.  First assume that the graph $\Ga_1$ is
isomorphic to $\Ga_2$.  Let $f:V_1\to V_2$ be an isomorphism of these
graphs and $\phi:E_1\to E_2$ be its companion mapping.  Clearly,
$f(v_1)=v_2$.  It is easy to see that any isomorphism of graphs preserves
distances between vertices.  It follows that $f$ maps $V_1(n)$ onto
$V_2(n)$ for any $n\ge0$.  Consequently, $\phi$ maps $E_1(n)$ onto
$E_2(n)$.  Hence the restriction of $f$ to the set $V_1(n)$ is an
isomorphism of the graphs $\oB_{\Ga_1}(v_1,n)$ and $\oB_{\Ga_2}(v_2,n)$.

Now assume that for every integer $n\ge0$ the closed balls
$\oB_{\Ga_1}(v_1,n)$ and $\oB_{\Ga_2}(v_2,n)$ are isomorphic.  Let
$f_n:V_1(n)\to V_2(n)$ be an isomorphism of these graphs and
$\phi_n:E_1(n)\to E_2(n)$ be its companion mapping.  Clearly,
$f_n(v_1)=v_2$.  Note that the closed ball $\oB_{\Ga_i}(v_i,n)$ is also the
closed ball with the same center and radius in any of the graphs
$\oB_{\Ga_i}(v_i,m)$, $m>n$.  It follows that the restriction of the
mapping $f_m$ to the set $V_1(n)$ is an isomorphism of the graphs
$\oB_{\Ga_1}(v_1,n)$ and $\oB_{\Ga_2}(v_2,n)$ while the restriction of
$\phi_m$ to $E_1(n)$ is its companion mapping.  Since the graphs $\Ga_1$
and $\Ga_2$ are locally finite, the sets $V_1(n),V_2(n),E_1(n),E_2(n)$ are
finite.  Hence there are only finitely many distinct restrictions
$f_m|_{V_1(n)}$ or $\phi_m|_{E_1(n)}$ for any fixed $n$.  Therefore one can
find nested infinite sets of indices $I_0\supset I_1\supset
I_2\supset\dots$ such that the restriction $f_m|_{V_1(n)}$ is the same for
all $m\in I_n$ and the restriction $\phi_m|_{E_1(n)}$ is the same for all
$m\in I_n$.  For any integer $n\ge0$ let $f'_n=f_m|_{V_1(n)}$ and
$\phi'_n=\phi_m|_{E_1(n)}$, where $m\in I_n$.  By construction, $f'_n$ is a
restriction of $f'_k$ and $\phi'_n$ is a restriction of $\phi'_k$ whenever
$n<k$.  Hence there exist maps $f:V_1\to V_2$ and $\phi:E_1\to E_2$ such
that all $f'_n$ are restrictions of $f$ and all $\phi'_n$ are restrictions
of $\phi$.  Since the graphs $\Ga_1$ and $\Ga_2$ are connected, any finite
collection of vertices and edges in either graph is contained in a closed
ball centered at the marked vertex.  As for any $n\ge0$ the mapping $f'_n$
is an isomorphism of $\oB_{\Ga_1}(v_1,n)$ and $\oB_{\Ga_2}(v_2,n)$ and
$\phi'_n$ is its companion mapping, it follows that $f$ is an isomorphism
of $\Ga_1$ and $\Ga_2$ and $\phi$ is its companion mapping.
\end{proof}

Lemma \ref{graph2} implies that $\de$ is a well-defined function on
$\MG_0\times\MG_0$.  This is a distance function, which makes $\MG_0$ into
an ultrametric space.

\begin{lemma}\label{graph3}
The distance function $\de$ is compatible with the topology on $\MG_0$.
\end{lemma}

\begin{proof}
The base of the topology on $\MG_0$ consists of the sets $\cU(\Ga_0,V_0)$.
The base of the topology defined by the distance function $\de$ is formed
by open balls $\cB(\Ga_1,\eps)=\{\Ga\in\MG_0\mid\de(\Ga,\Ga_1)<\eps\}$,
where $\Ga_1$ can be any graph in $\MG_0$ and $\eps>0$.  We have to show
that any element of either base is the union of some elements of the other
base.

First consider an open ball $\cB(\Ga_1,\eps)$.  If $\eps>1$ then
$\cB(\Ga_1,\eps)=\MG_0$, which is the union of all sets $\cU(\Ga_0,V_0)$.
Otherwise let $n$ be the largest integer such that $\eps\le2^{-n}$.
Clearly, $\cB(\Ga_1,\eps)=\cB(\Ga_1,2^{-n})$.  Let
$\Ga_0=\oB_{\Ga_1}(v_1,n)$, where $v_1$ is the marked vertex of the graph
$\Ga_1$, and let $V_0$ be the set of all vertices of $\Ga_1$ at distance at
most $n-1$ from $v_1$.  Consider an arbitrary graph $\Ga\in\MG_0$ such that
$\Ga_0$ is a subgraph of $\Ga$.  Clearly, $\Ga_0$ is also a subgraph of the
closed ball $\oB_\Ga(v_1,n)$.  If $v$ is a vertex of $\Ga$ at distance $k$
from the marked vertex $v_1$, then any vertex joined to $v$ by an edge is
at distance at most $k+1$ and at least $k-1$ from $v_1$.  Moreover, if
$k>0$ then $v$ is joined to a vertex at distance exactly $k-1$ from $v_1$.
It follows that $\Ga_0=\oB_\Ga(v_1,n)$ if and only if no vertex from the
set $V_0$ is joined in $\Ga$ to a vertex that is not a vertex of $\Ga_0$.
Thus $\cB(\Ga_1,2^{-n})=\cU(\Ga_0,V_0)$.

Now consider the set $\cU(\Ga_0,V_0)$, where $\Ga_0$ is a finite graph in
$\MG_0$ and $V_0$ is a subset of its vertex set.  Denote by $v_0$ the
marked vertex of $\Ga_0$.  Let $n$ be the smallest integer such that every
vertex of $\Ga_0$ is at distance at most $n$ from $v_0$ and every vertex
from $V_0$ is at distance at most $n-1$ from $v_0$.  Take any graph
$\Ga\in\MG_0$ such that $\Ga_0$ is a subgraph of $\Ga$ and there is no edge
in $\Ga$ joining a vertex from $V_0$ to a vertex outside the vertex set of
$\Ga_0$.  Let $\Ga_1=\oB_\Ga(v_0,n)$ and $V_1$ be the set of all vertices
of $\Ga$ at distance at most $n-1$ from $v_0$.  By the above,
$\cU(\Ga_1,V_1)=\cB(\Ga,2^{-n})$.  At the same time,
$\cU(\Ga_1,V_1)\subset\cU(\Ga_0,V_0)$ since $\Ga_0$ is a subgraph of
$\Ga_1$ and $V_0$ is a subset of $V_1$.  Thus for any graph
$\Ga\in\cU(\Ga_0,V_0)$ the entire open ball $\cB(\Ga,2^{-n})$ is contained
in $\cU(\Ga_0,V_0)$.  In particular, $\cU(\Ga_0,V_0)$ is the union of those
open balls.
\end{proof}

Given a positive integer $N$ and a finite set $L$, let $\MG(N,L)$ denote
the subset of $\MG$ consisting of all graphs in which every vertex is the
endpoint for at most $N$ edges and every label belongs to $L$.  Further,
let $\MG_0(N,L)=\MG(N,L)\cap\MG_0$.

\begin{proposition}\label{graph4}
$\MG_0(N,L)$ is a compact subset of the metric space $\MG_0$.
\end{proposition}

\begin{proof}
We have to show that any sequence of graphs $\Ga_1,\Ga_2,\dots$ in
$\MG_0(N,L)$ has a subsequence converging to some graph in $\MG_0(N,L)$.
For any positive integer $n$ let $V_n$ denote the vertex set of the graph
$\Ga_n$, $E_n$ denote its sets of edges, and $v_n$ denote the marked vertex
of $\Ga_n$.  First consider the special case when each $\Ga_n$ is a
subgraph of $\Ga_{n+1}$.  Let $\Ga$ be the graph with the vertex set
$V=V_1\cup V_2\cup\dots$ and the set of edges $E=E_1\cup E_2\cup\dots$.  We
assume that any edge $e\in E_n$ retains its attributes (beginning, end, and
label) in the graph $\Ga$.  The common marked vertex of the graphs $\Ga_n$
is set as the marked vertex of $\Ga$.  Note that any finite collection of
vertices and edges of the graph $\Ga$ is already contained in some $\Ga_n$.
As the graphs $\Ga_1,\Ga_2,\dots$ belong to $\MG_0(N,L)$, it follows that
$\Ga\in\MG_0(N,L)$ as well.  In particular, for any integer $k\ge0$ the
closed ball $\oB_\Ga(v_1,k)$ is a finite graph.  Then it is a subgraph of
some $\Ga_n$.  Clearly, $\oB_\Ga(v_1,k)$ is also a subgraph of the graphs
$\Ga_{n+1},\Ga_{n+2},\dots$.  Moreover, it remains the closed ball of
radius $k$ centered at the marked vertex in all these graphs.  It follows
that $\de(\Ga_m,\Ga)<2^{-k}$ for $m\ge n$.  Since $k$ can be arbitrarily
large, the sequence $\Ga_1,\Ga_2,\dots$ converges to $\Ga$ in the metric
space $\MG_0$.

Next consider a more general case when each $\Ga_n$ is isomorphic to a
subgraph of $\Ga_{n+1}$.  This case is reduced to the previous one by
repeatedly using the following observation: if a graph $P_0$ is isomorphic
to a subgraph of a graph $P$ then there exists a graph $P'$ isomorphic to
$P$ such that $P_0$ is a subgraph of $P'$.

Finally consider the general case.  For any graph in $\MG_0(N,L)$, the
closed ball of radius $k$ with any center contains at most
$1+N+N^2+\cdots+N^{k-1}$ vertices while the number of edges is at most $N$
times the number of vertices.  Hence for any fixed $k$ the number of
vertices and edges in the balls $\oB_{\Ga_n}(v_n,k)$ is uniformly bounded,
which implies that there are only finitely many non-isomorphic graphs among
them.  Therefore one can find nested infinite sets of indices $I_0\supset
I_1\supset I_2\supset\dots$ such that the closed balls $\oB_{\Ga_n}(v_n,k)$
are isomorphic for all $n\in I_k$.  Choose an increasing sequence of
indices $n_0,n_1,n_2,\dots$ such that $n_k\in I_k$ for all $k$, and let
$\Ga'_k$ be the closed ball of radius $k$ in the graph $\Ga_{n_k}$ centered
at the marked point $v_{n_k}$.  Clearly, $\Ga'_k\in\MG_0(N,L)$ and
$\de(\Ga'_k,\Ga_{n_k})<2^{-k}$.  By construction, $\Ga'_k$ is isomorphic to
a subgraph of $\Ga'_m$ whenever $k<m$.  By the above the sequence
$\Ga'_0,\Ga'_1,\Ga'_2,\dots$ converges to a graph $\Ga\in\MG_0(N,L)$.
Since $\de(\Ga'_k,\Ga_{n_k})<2^{-k}$ for all $k\ge0$, the subsequence
$\Ga_{n_0},\Ga_{n_1},\Ga_{n_2},\dots$ converges to the graph $\Ga$ as well.
\end{proof}

\section{Group actions}\label{act}

Let $M$ be an arbitrary nonempty set.  Invertible transformations
$\phi:M\to M$ form a transformation group.  An {\em action\/} $A$ of an
abstract group $G$ on the set $M$ is a homomorphism of $G$ into that
transformation group.  The action can be regarded as a collection of
invertible transformations $A_g:M\to M$, $g\in G$, where $A_g$ is the image
of $g$ under the homomorphism.  The transformations are to satisfy
$A_gA_h=A_{gh}$ for all $g,h\in G$.  We say that $A_g$ is the action of an
element $g$ within the action $A$.  Alternatively, the action of the group
$G$ can be given as a mapping $A:G\times M\to M$ such that $A(g,x)=A_g(x)$
for all $g\in G$ and $x\in M$.  Such a mapping defines an action of $G$ if
and only if the following two conditions hold:
\begin{itemize}
\item $A(gh,x)=A(g,A(h,x))$ for all $g,h\in G$ and $x\in M$;
\item $A(1_G,x)=x$ for all $x\in M$, where $1_G$ is the unity of the group
$G$.
\end{itemize}

A nonempty set $S\subset G$ is called a {\em generating set\/} for the
group $G$ if any element $g\in G$ can be represented as a product
$g_1g_2\dots g_k$ where each factor $g_i$ is an element of $S$ or the
inverse of an element of $S$.  The elements of the generating set are
called {\em generators\/} of the group $G$.  The generating set $S$ is {\em
symmetric\/} if it is closed under taking inverses, i.e., $s^{-1}\in S$
whenever $s\in S$.  If $S$ is a generating set for $G$ then any action $A$
of the group $G$ is uniquely determined by transformations $A_s$, $s\in S$.

Suppose $G$ is a topological group.  An action of $G$ on a topological
space $M$ is a {\em continuous action\/} if it is continuous as a mapping
of $G\times M$ to $M$.  Similarly, an action of $G$ on a measured space $M$
is a {\em measurable action\/} if it is measurable as a mapping of
$G\times M$ to $M$.  A measurable action $A$ of the group $G$ on a measured
space $M$ with a measure $\mu$ is {\em measure-preserving\/} if the action
of every element of $G$ is measure-preserving, i.e.,
$\mu\bigl(A_g^{-1}(W)\bigr)=\mu(W)$ for all $g\in G$ and measurable sets
$W\subset M$.  In what follows, the group $G$ will be a discrete countable
group.  In that case, an action $A$ of $G$ is continuous if and only if all
transformations $A_g$, $g\in G$ are continuous.  Likewise, the action $A$
is measurable if and only if every $A_g$ is measurable.

Given an action $A$ of a group $G$ on a set $M$, the {\em orbit\/} $O_A(x)$
of a point $x\in M$ under the action $A$ is the set of all points $A_g(x)$,
$g\in G$.  A subset $M_0\subset M$ is {\em invariant\/} under the action
$A$ if $A_g(M_0)\subset M_0$ for all $g\in G$.  Clearly, the orbit $O_A(x)$
is invariant under the action.  Moreover, this is the smallest invariant
set containing $x$.  The {\em restriction\/} of the action $A$ to a
nonempty invariant set $M_0$ is an action of $G$ obtained by restricting
every transformation $A_g$ to $M_0$.  Equivalently, one might restrict the
mapping $A:G\times M\to M$ to the set $G\times M_0$.  The action $A$ is
{\em transitive\/} if the only invariant subsets of $M$ are the empty set
and $M$ itself.  Equivalently, the orbit of any point is the entire set
$M$.  Assuming the action $A$ is continuous, it is {\em topologically
transitive\/} if there is an orbit dense in $M$, and {\em minimal\/} if
every orbit of $A$ is dense.  The action is minimal if and only if the
empty set and $M$ are the only closed invariant subsets of $M$.  Assuming
the action $A$ is measure-preserving, it is {\em ergodic\/} if any
measurable invariant subset of $M$ has zero or full measure.  A continuous
action on a compact space $M$ is {\em uniquely ergodic\/} if there exists a
unique Borel probability measure on $M$ invariant under the action (the
action is going to be ergodic with respect to that measure).

Given an action $A$ of a group $G$ on a set $M$, the {\em stabilizer\/}
$\St_A(x)$ of a point $x\in M$ under the action $A$ is the set of all
elements $g\in G$ whose action fixes $x$, i.e., $A_g(x)=x$.  The stabilizer
$\St_A(x)$ is a subgroup of $G$.  The action is {\em free\/} if all
stabilizers are trivial.  In the case when the action $A$ is continuous, we
define the {\em neighborhood stabilizer\/} $\St^o_A(x)$ of a point $x\in M$
to be the set of all $g\in G$ whose action fixes the point $x$ along with
its neighborhood (the neighborhood may depend on $g$).  The neighborhood
stabilizer $\St^o_A(x)$ is a normal subgroup of $\St_A(x)$.

Let $A:G\times M_1\to M_1$ and $B:G\times M_2\to M_2$ be actions of a group
$G$ on sets $M_1$ and $M_2$, respectively.  The actions $A$ and $B$ are
{\em conjugated\/} if there exists a bijection $f:M_1\to M_2$ such that
$B_g=fA_gf^{-1}$ for all $g\in G$.  An equivalent condition is that
$A(g,x)=B(g,f(x))$ for all $g\in G$ and $x\in M_1$.  The bijection $f$ is
called a {\em conjugacy\/} of the action $A$ with $B$.  Two continuous
actions of the same group are {\em continuously conjugated\/} if they are
conjugated and, moreover, the conjugacy can be chosen to be a
homeomorphism.  Similarly, two measurable actions are {\em measurably
conjugated\/} if they are conjugated and, moreover, the conjugacy $f$ can
be chosen so that both $f$ and the inverse $f^{-1}$ are measurable.  Also,
two measure-preserving actions are {\em isomorphic\/} if they are
conjugated and, moreover, the conjugacy can be chosen to be an isomorphism
of spaces with measure.  The measure-preserving actions are {\em isomorphic
modulo zero measure\/} if each action admits an invariant set of full
measure such that the corresponding restrictions are isomorphic.

Given two actions $A:G\times M_1\to M_1$ and $B:G\times M_2\to M_2$ of a
group $G$, the action $A$ is an {\em extension\/} of $B$ if there exists a
mapping $f$ of $M_1$ onto $M_2$ such that $B_gf=fA_g$ for all $g\in G$.
The extension is {\em $k$-to-$1$} if $f$ is $k$-to-$1$.  The extension is
{\em continuous\/} if the actions $A$ and $B$ are continuous and $f$ can be
chosen continuous.

\section{The Schreier graphs}\label{sch}

Let $G$ be a finitely generated group.  Let us fix a finite symmetric
generating set $S$ for $G$.  Given an action $A$ of the group $G$ on a set
$M$, the {\em Schreier graph\/} $\Ga_{\Sch}(G,S;A)$ of the action relative
to the generating set $S$ is a directed graph with labeled edges.  The
vertex set of the graph $\Ga_{\Sch}(G,S;A)$ is $M$, the set of edges is
$M\times S$, and the set of labels is $S$.  For any $x\in M$ and $s\in S$
the edge $(x,s)$ has beginning $x$, end $A_s(x)$, and carries label $s$.
Clearly, the action $A$ can be uniquely recovered from its Schreier graph.
Given another action $A'$ of $G$ on some set $M'$, the Schreier graph
$\Ga_{\Sch}(G,S;A')$ is isomorphic to $\Ga_{\Sch}(G,S;A)$ if and only if
the actions $A$ and $A'$ are conjugated.  Indeed, a bijection $f:M\to M'$
is an isomorphism of the Schreier graphs if and only if $A'_s=fA_sf^{-1}$
for all $s\in S$, which is equivalent to $f$ being a conjugacy of the
action $A$ with $A'$.

Any graph of the form $\Ga_{\Sch}(G,S;A)$ is called a {\em Schreier
graph\/} of the group $G$ (relative to the generating set $S$).  Notice
that any graph isomorphic to a Schreier graph is also a Schreier graph up
to renaming edges.  This follows from the next proposition, which explains
how to recognize a Schreier graph of $G$.

\begin{proposition}\label{sch1}
A directed graph $\Ga$ with labeled edges is, up to renaming edges, a
Schreier graph of the group $G$ relative to the generating set $S$ if and
only if the following conditions are satisfied:
\begin{itemize}
\item[(i)]
all labels are in $S$;
\item[(ii)]
for any vertex $v$ and any generator $s\in S$ there exists a unique edge
with beginning $v$ and label $s$;
\item[(iii)]
given a directed path with code word $s_1s_2\dots s_k$, the path is closed
whenever the reversed code word $s_k\dots s_2s_1$ equals $1_G$ when
regarded as a product in $G$.
\end{itemize}
\end{proposition}

\begin{proof}
First suppose $\Ga$ is a Schreier graph $\Ga_{\Sch}(G,S;A)$.  Consider an
arbitrary directed path in the graph $\Ga$.  Let $v$ be the beginning of
the path and $s_1s_2\dots s_k$ be its code word.  Then the consecutive
vertices of the path are $v_0=v,v_1,\dots,v_k$, where
$v_i=A_{s_i}(v_{i-1})$ for $1\le i\le k$.  Hence the end of the path is
$A_{s_k}\dots A_{s_2}A_{s_1}(v)=A_g(v)$, where $g$ denotes
$s_k\dots s_2s_1$ regarded as a product in $G$.  Clearly, the path is
closed whenever $g=1_G$.  Thus any Schreier graph of the group $G$
satisfies the condition (iii).  The conditions (i) and (ii) are trivially
satisfied as well.  It is easy to see that the conditions (i), (ii), and
(iii) are preserved under isomorphisms of graphs.  In particular, they hold
for any graph that coincides with a Schreier graph up to renaming edges.

Now suppose $\Ga$ is a directed graph with labeled edges that satisfies the
conditions (i), (ii), and (iii).  Let $M$ denote the vertex set of $\Ga$.
Given a word $w=s_1s_2\dots s_k$ over the alphabet $S$, we define a
transformation $B_w:M\to M$ as follows.  The condition (ii) implies that
for any vertex $v\in M$ there is a unique directed path in $\Ga$ with
beginning $v$ and code word $s_k\dots s_2s_1$ (the word $w$ reversed).  We
set $B_w(v)$ to be the end of that path.  For any words $w=s_1s_2\dots s_k$
and $w'=s'_1s'_2\dots s'_m$ over the alphabet $S$ let $ww'$ denote the
concatenated word $s_1s_2\dots s_ks'_1s'_2\dots s'_m$.  Then $B_{ww'}(v)
=B_w(B_{w'}(v))$ for all $v\in M$.  Any word over the alphabet $S$ can be
regarded as a product in the group $G$ thus representing an element $g\in
G$.  Clearly, the concatenation of words corresponds to the multiplication
in the group.  The condition (iii) means that $B_w$ is the identity
transformation whenever the word $w$ represents the unity $1_G$.  This
implies that transformations $B_w$ and $B_{w'}$ are the same if the words
$w$ and $w'$ represent the same element $g\in G$.  Indeed, let
$w=s_1s_2\dots s_k$, $w'=s'_1s'_2\dots s'_m$ and consider the third word
$z=s_k^{-1}\dots s_2^{-1}s_1^{-1}$.  The word $z$ represents the inverse
$g^{-1}$.  Therefore the words $wz$ and $zw'$ both represent the unity.
Then $B_w=B_wB_{zw'}=B_{wzw'}=B_{wz}B_{w'}=B_{w'}$.  Now for any $g\in G$
we let $A_g=B_w$, where $w$ is an arbitrary word over the alphabet $S$
representing $g$.  By the above $A_g$ is a well-defined transformation of
$M$, $A_{1_G}$ is the identity transformation, and $A_{gg'}=A_gA_{g'}$ for
all $g,g'\in G$.  Hence the transformations $A_g$, $g\in G$ constitute an
action $A$ of the group $G$ on the vertex set $M$.  By construction, for
any $v\in M$ and $s\in S$ the vertex $A_s(v)$ is the end of the edge with
beginning $v$ and label $s$.  In view of the conditions (i) and (ii), this
means that the graph $\Ga$ coincides with the Schreier graph
$\Ga_{\Sch}(G,S;A)$ up to renaming edges.
\end{proof}

For any $x\in M$ let $\Ga_{\Sch}(G,S;A,x)$ denote the Schreier graph of the
restriction of the action $A$ to the orbit of $x$.  We refer to
$\Ga_{\Sch}(G,S;A,x)$ as the Schreier graph of the orbit of $x$.  It is
easy to observe that $\Ga_{\Sch}(G,S;A,x)$ is the connected component of
the graph $\Ga_{\Sch}(G,S;A)$ containing the vertex $x$.  In particular,
the Schreier graph of the action $A$ is connected if and only if the action
is transitive, in which case $\Ga_{\Sch}(G,S;A,x)=\Ga_{\Sch}(G,S;A)$ for
all $x\in M$.  Let $\Ga_{\Sch}^*(G,S;A,x)$ denote a marked graph obtained
from $\Ga_{\Sch}(G,S;A,x)$ by marking the vertex $x$.  We refer to it as
the {\em marked Schreier graph\/} of the point $x$ (under the action $A$).
Notice that the point $x$ and the restriction of the action $A$ to its
orbit are uniquely recovered from the graph $\Ga_{\Sch}^*(G,S;A,x)$.  Any
graph of the form $\Ga_{\Sch}^*(G,S;A,x)$ is called a marked Schreier graph
of the group $G$ (relative to the generating set $S$).

Let $\Sch(G,S)$ denote the set of isomorphism classes of all marked
Schreier graphs of the group $G$ relative to the generating set $S$.  A
graph $\Ga\in\MG$ belongs to $\Sch(G,S)$ if it is a marked directed graph
that is connected and satisfies conditions (i), (ii), (iii) of Proposition
\ref{sch1}.

The group $G$ acts naturally on the set of the marked Schreier graphs of
$G$ by changing the marked vertex.  The action $\cA$ is given by
$\cA_g\bigl(\Ga_{\Sch}^*(G,S;A,x)\bigr)=\Ga_{\Sch}^*(G,S;A,A_g(x))$,
$g\in G$.  It turns out that $\cA$ is well defined as an action on
$\Sch(G,S)$.  Indeed, let $\Ga_{\Sch}^*(G,S;B,y)$ be a marked Schreier
graph isomorphic to $\Ga_{\Sch}^*(G,S;A,x)$.  Then any isomorphism $f$ of
the latter graph with the former one is simultaneously a conjugacy of the
restriction of the action $A$ to the orbit of $x$ with the restriction of
the action $B$ to the orbit of $y$.  Since $f(x)=y$, it follows that
$f(A_g(x))=B_g(y)$ for all $g\in G$.  Hence for any $g\in G$ the map $f$ is
also an isomorphism of the graph $\Ga_{\Sch}^*(G,S;A,A_g(x))$ with
$\Ga_{\Sch}^*(G,S;B,B_g(y))$.

\begin{proposition}\label{sch2}
$\Sch(G,S)$ is a compact subset of the metric space $\MG_0$.  The action of
the group $G$ (regarded as a discrete group) on $\Sch(G,S)$ is continuous.
\end{proposition}

\begin{proof}
Let $N$ be the number of elements in the generating set $S$.  Then every
vertex $v$ of a graph $\Ga\in\Sch(G,S)$ is the beginning of exactly $N$
edges.  Furthermore, $v$ is the end of an edge with beginning $v'$ and
label $s$ if and only if $v'$ is the end of the edge with beginning $v$ and
label $s^{-1}$.  It follows that $v$ is also the end of exactly $N$ edges.
Hence any vertex of $\Ga$ is an endpoint for at most $2N$ edges.  Therefore
$\Sch(G,S)\subset\MG(2N,S)$.  Since all marked Schreier graphs are
connected, we have $\Sch(G,S)\subset\MG_0(2N,S)\subset\MG_0$.

Now let us show that the set $\Sch(G,S)$ is closed in the topological space
$\MG_0$.  Take any graph $\Ga\in\MG_0$ not in that set.  Then $\Ga$ does
not satisfy at least one of the conditions (i), (ii), and (iii) in
Proposition \ref{sch1}.  First consider the case when the condition (i) or
(iii) does not hold.  Since the graph $\Ga$ is locally finite, it has a
finite subgraph $\Ga_0$ for which the same condition does not hold.  Since
$\Ga$ is connected, we can choose the subgraph $\Ga_0$ to be marked and
connected so that $\Ga_0\in\MG_0$.  Clearly, the same condition does not
hold for any graph $\Ga'$ such that $\Ga_0$ is a subgraph of $\Ga'$.  It
follows that the neighborhood $\cU(\Ga_0,\emptyset)$ of the graph $\Ga$ is
disjoint from $\Sch(G,S)$.  Next consider the case when $\Ga$ does not
satisfy the condition (ii).  Let $v$ be the vertex of $\Ga$ such that for
some generator $s\in S$ there are either several edges with beginning $v$
and label $s$ or no such edges at all.  Since $\Ga\in\MG_0$, there exists a
finite connected subgraph $\Ga_0$ of $\Ga$ that contains the marked vertex,
the vertex $v$, and all edges for which $v$ is an endpoint.  Then
$\Ga_0\in\MG_0$ and the open set $\cU(\Ga_0,\{v\})$ is a neighborhood of
$\Ga$.  By construction, the condition (ii) fails in the entire
neighborhood so that $\cU(\Ga_0,\{v\})$ is disjoint from $\Sch(G,S)$.  Thus
the set $\MG_0\setminus\Sch(G,S)$ is open in $\MG_0$.  Therefore the set
$\Sch(G,S)$ is closed.

Since the closed set $\Sch(G,S)$ is contained in $\MG_0(2N,S)$, which is a
compact set due to Proposition \ref{graph4}, the set $\Sch(G,S)$ is compact
as well.

An action of the group $G$ is continuous whenever the generators act
continuously.  To prove that the transformations $\cA_s$, $s\in S$ are
continuous, we are going to show that $\de(\cA_s(\Ga),\cA_s(\Ga'))\le
2\de(\Ga,\Ga')$ for any graphs $\Ga,\Ga'\in\Sch(G,S)$ and any generator
$s\in S$.  If the graphs $\Ga$ and $\Ga'$ are isomorphic, then the graphs
$\cA_s(\Ga)$ and $\cA_s(\Ga')$ are also isomorphic so that
$\de(\cA_s(\Ga),\cA_s(\Ga'))=\de(\Ga,\Ga')=0$.  Otherwise
$\de(\Ga,\Ga')=2^{-n}$ for some nonnegative integer $n$.  Since the
distance between any graphs in $\MG_0$ never exceeds $1$, it is enough to
consider the case $n\ge2$.  Let $v$ denote the marked vertex of $\Ga$ and
$v'$ denote the marked vertex of $\Ga'$.  By definition of the distance
function, the closed balls $\oB_\Ga(v,n-1)$ and $\oB_{\Ga'}(v',n-1)$ are
isomorphic.  Consider an isomorphism $f$ of these graphs.  Clearly,
$f(v)=v'$.  Let $v_1$ denote the marked vertex of the graph $\cA_s(\Ga)$
and $v'_1$ denote the marked vertex of $\cA_s(\Ga')$.  Then $v_1$ is the
end of the edge with beginning $v$ and label $s$ in the graph $\Ga$.
Similarly, $v'_1$ is the end of the edge with beginning $v'$ and label $s$
in $\Ga'$.  It follows that $f(v_1)=v'_1$.  Since the vertex $v_1$ is
joined to $v$ by an edge, the closed ball $\oB_\Ga(v_1,n-2)$ is a subgraph
of $\oB_\Ga(v,n-1)$.  Note that $\oB_\Ga(v_1,n-2)$ remains the closed ball
with the same center and radius in the graph $\oB_\Ga(v,n-1)$.  Similarly,
$\oB_{\Ga'}(v'_1,n-2)$ is a subgraph of $\oB_{\Ga'}(v,n-1)$ and it is also
the closed ball of radius $n-2$ centered at $v'_1$ in the graph
$\oB_{\Ga'}(v,n-1)$.  Since $f(v_1)=v'_1$ and any isomorphism of graphs
preserves distance between vertices, the restriction $f_0$ of $f$ to the
vertex set of $\oB_\Ga(v_1,n-2)$ is an isomorphisms of the graphs
$\oB_\Ga(v_1,n-2)$ and $\oB_{\Ga'}(v'_1,n-2)$.  It remains to notice that
the closed ball $\oB_{\cA_s(\Ga)}(v_1,n-2)$ differs from $\oB_\Ga(v_1,n-2)$
in that the marked vertex is $v_1$ and, similarly,
$\oB_{\cA_s(\Ga')}(v'_1,n-2)$ differs from $\oB_{\Ga'}(v'_1,n-2)$ in that
the marked vertex is $v'_1$.  Therefore $f_0$ is also an isomorphism of
$\oB_{\cA_s(\Ga)}(v_1,n-2)$ and $\oB_{\cA_s(\Ga')}(v'_1,n-2)$.  By
definition of the distance function, $\de(\cA_s(\Ga),\cA_s(\Ga'))\le
2^{-(n-1)}=2\de(\Ga,\Ga')$.
\end{proof}

Let $A$ be an action of the group $G$ on a set $M$.  To any point $x\in M$
we associate three subgroups of $G$: the stabilizer $\St_A(x)$ of $x$, the
stabilizer $\St_{\cA}(\Ga^*_x)$ of the marked Schreier graph
$\Ga^*_x=\Ga_{\Sch}^*(G,S;A,x)$, and the neighborhood stabilizer
$\St_{\cA}^o(\Ga^*_x)$ (if the action $A$ is continuous then there is the
fourth subgroup, the neighborhood stabilizer of $x$).  Clearly, the graph
$\Ga^*_{A_g(x)}$ coincides with $\Ga^*_x$ if and only if $A_g(x)=x$.
However this does not imply that the stabilizer of the graph is the same as
the stabilizer of $x$.  Since $\cA$ is an action on isomorphism classes of
graphs, we have $g\in\St_{\cA}(\Ga^*_x)$ if and only if the graph
$\Ga^*_{A_g(x)}$ is isomorphic to $\Ga^*_x$.

\begin{lemma}\label{sch3}
(i) $\St_A(x)$ is a normal subgroup of
$\St_{\cA}\bigl(\Ga_{\Sch}^*(G,S;A,x)\bigr)$.

(ii) The quotient of $\St_{\cA}\bigl(\Ga_{\Sch}^*(G,S;A,x)\bigr)$ by
$\St_A(x)$ is isomorphic to the group of all automorphisms of the unmarked
graph $\Ga_{\Sch}(G,S;A,x)$.

(iii) $\St_A(x)$ is a subgroup of
$\St_{\cA}^o\bigl(\Ga_{\Sch}^*(G,S;A,x)\bigr)$.
\end{lemma}

\begin{proof}
Without loss of generality we can assume that the action $A$ is transitive.
For brevity, let $\Ga^*$ denote the marked graph $\Ga_{\Sch}^*(G,S;A,x)$,
$\Ga$ denote the unmarked graph $\Ga_{\Sch}(G,S;A,x)$, and $R$ denote the
group of all automorphisms of $\Ga$.  Consider an arbitrary $f\in R$.  For
any vertex $y\in O_A(x)$ and any label $s\in S$ the unique edge of $\Ga$
with beginning $y$ and label $s$ has end $A_s(y)$.  It follows that
$f(A_s(y))=A_s(f(y))$.  Since the action $A$ is transitive, the
automorphism $f$ commutes with transformations $A_s$, $s\in S$.  Then $f$
commutes with $A_g$ for all $g\in G$.  Notice that the automorphism $f$ is
uniquely determined by the vertex $f(x)$.  Indeed, any vertex $y$ of $\Ga$
is represented as $A_g(x)$ for some $g\in G$, then
$f(y)=f(A_g(x))=A_g(f(x))$.  In particular, $f$ is the identity if
$f(x)=x$.

To prove the statements (i) and (ii), we are going to construct a
homomorphism $\Psi$ of the stabilizer $\St_{\cA}(\Ga^*)$ onto the group $R$
with kernel $\St_A(x)$.  An element $g\in G$ belongs to $\St_{\cA}(\Ga^*)$
if the graph $\Ga^*$ is isomorphic to $\Ga_{\Sch}^*(G,S;A,A_g(x))$.  An
isomorphism of these marked graphs is an automorphism of the unmarked graph
$\Ga$ that sends $x$ to $A_g(x)$.  Hence $g\in\St_{\cA}(\Ga^*)$ if and only
if $A_g(x)=\psi_g(x)$ for some $\psi_g\in R$.  By the above the
automorphism $\psi_g$ is uniquely determined by $A_g(x)$.  Now we define a
mapping $\Psi:\St_{\cA}(\Ga^*)\to R$ by $\Psi(g)=\psi_{g^{-1}}$.  It is
easy to observe that $\Psi$ maps $\St_{\cA}(\Ga^*)$ onto $R$ and the
preimage of the identity under $\Psi$ is $\St_A(x)$.  Further, for any
$g,h\in\St_{\cA}(\Ga^*)$ we have $\psi_{(gh)^{-1}}(x)=A_{gh}^{-1}(x)
=A_h^{-1}(A_g^{-1}(x))=A_h^{-1}(\psi_{g^{-1}}(x))$.  Recall that the
automorphism $\psi_{g^{-1}}$ commutes with the action $A$, in particular,
$A_h^{-1}\psi_{g^{-1}}=\psi_{g^{-1}}A_h^{-1}$.  Then $\psi_{(gh)^{-1}}(x)
=\psi_{g^{-1}}(A_h^{-1}(x))=\psi_{g^{-1}}(\psi_{h^{-1}}(x))$, which implies
that $\Psi(gh)=\Psi(g)\Psi(h)$.  Thus $\Psi$ is a homomorphism.

We proceed to the statement (iii).  Take any element $g\in\St_A(x)$.  It
can be represented as a product $s_1s_2\dots s_k$, where each $s_i$ is in
$S$.  Let $\ga$ denote the unique directed path in $\Ga^*$ with beginning
$x$ and code word $s_k\dots s_2s_1$.  By construction, the end of the path
$\ga$ is $A_g(x)$ so that the path is closed.  Let $\Ga_0^*$ denote the
subgraph of $\Ga^*$ whose vertex set consists of all vertices of the path
$\ga$.  Clearly, $\Ga_0^*$ is a marked graph, finite and connected.  Hence
$\Ga_0^*\in\MG_0$.  Any graph $\Ga_1^*\in\cU(\Ga_0^*,\emptyset)$ admits a
closed directed path with beginning at the marked point and code word
$s_k\dots s_2s_1$.  If $\Ga_1^*=\Ga_{\Sch}^*(G,S;B,y)$, this implies that
$B_g(y)=y$.  Hence $g\in\St_B(y)\subset
\St_{\cA}\bigl(\Ga_{\Sch}^*(G,S;B,y)\bigr)$.  Thus the transformation
$\cA_g$ fixes the set $\cU(\Ga_0^*,\emptyset)\cap\Sch(G,S)$, which is an
open neighborhood of the graph $\Ga^*$ in $\Sch(G,S)$.
\end{proof}

Any group $G$ acts naturally on itself by left multiplication.  The action
$\adj_G:G\times G\to G$, called {\em adjoint}, is given by
$\adj_G(g_0,g)=g_0g$.  The Schreier graph of this action relative to any
generating set $S$ is the {\em Cayley graph\/} of the group $G$ relative to
$S$.  Given a subgroup $H$ of $G$, the adjoint action of the group $G$
descends to an action on $G/H$.  The action $\adj_{G,H}:G\times G/H\to G/H$
is given by $\adj_{G,H}(g_0,gH)=(g_0g)H$.  The Schreier graph of the latter
action relative to a generating set $S$ is denoted $\Ga_{\coset}(G,S;H)$.
It is called a {\em Schreier coset graph}.  The {\em marked Schreier coset
graph\/} $\Ga_{\coset}^*(G,S;H)$ is the marked Schreier graph of the coset
$H$ under the action $\adj_{G,H}$.  It is obtained from
$\Ga_{\coset}(G,S;H)$ by marking the vertex $H$.

\begin{proposition}\label{sch4}
A marked Schreier graph $\Ga_{\Sch}^*(G,S;A,x)$ is isomorphic to a marked
Schreier coset graph $\Ga_{\coset}^*(G,S;H)$ if and only if $H=\St_A(x)$.
\end{proposition}

\begin{proof}
Let $H_0$ denote the stabilizer $\St_A(x)$.  Suppose
$A_{g_1}(x)=A_{g_2}(x)$ for some $g_1,g_2\in G$.  Then $A_{g_2^{-1}g_1}(x)
=A_{g_2}^{-1}(A_{g_1}(x))=x$ so that $g_2^{-1}g_1\in H_0$.  Hence
$g_2^{-1}g_1H_0=H_0$ and $g_1H_0=g_2H_0$.  Conversely, if $g_1H_0=g_2H_0$
then $g_1=g_2h$ for some $h\in H_0$.  It follows that $A_{g_1}(x)
=A_{g_2}(A_h(x))=A_{g_2}(x)$.

Let us define a mapping $f:G/H_0\to O_A(x)$ by $f(gH_0)=A_g(x)$.  By the
above $f$ is well defined and one-to-one.  Clearly, it maps $G/H_0$ onto
the entire orbit $O_A(x)$.  For any $g_0,g\in G$ we have $f(g_0gH_0)
=A_{g_0g}(x)=A_{g_0}(A_g(x))=A_{g_0}(f(gH_0))$.  Therefore $f$ is a
conjugacy of the action $\adj_{G,H_0}$ with the restriction of the action
$A$ to the orbit $O_A(x)$.  It follows that $f$ is also an isomorphism of
the unmarked graphs $\Ga_{\coset}(G,S;H_0)$ and $\Ga_{\Sch}(G,S;A,x)$.  As
$f(H_0)=x$, the mapping $f$ is an isomorphism of the marked graphs
$\Ga_{\coset}^*(G,S;H_0)$ and $\Ga_{\Sch}^*(G,S;A,x)$ as well.

Since any isomorphism of Schreier graphs of the group $G$ is also a
conjugacy of the corresponding actions, it preserves stabilizers of
vertices.  In particular, marked Schreier graphs cannot be isomorphic if
the stabilizers of their marked vertices do not coincide.  For any subgroup
$H$ of $G$ the stabilizer of the coset $H$ under the action $\adj_{G,H}$ is
$H$ itself.  Therefore the graph $\Ga_{\coset}^*(G,S;H)$ is not isomorphic
to $\Ga_{\Sch}^*(G,S;A,x)$ if $H\ne\St_A(x)$.
\end{proof}

\section{Space of subgroups}\label{sub}

Let $G$ be a discrete countable group.  Denote by $\Sub(G)$ the set of all
subgroups of $G$.  We endow the set $\Sub(G)$ with a topology as follows.
First we consider the product topology on $\{0,1\}^G$.  The set $\{0,1\}^G$
is in a one-to-one correspondence with the set of all functions
$f:G\to\{0,1\}$.  Also, any subset $H\subset G$ (in particular, any
subgroup) is assigned the indicator function $\chi_H:G\to\{0,1\}$ defined
by
$$
\chi_H(g)=
\left\{
\begin{array}{l}
1 \mbox{ if } g\in H,\\[1mm] 0 \mbox{ if } g\notin H.
\end{array}
\right.
$$
This gives rise to a mapping $j:\Sub(G)\to\{0,1\}^G$, which is an
embedding.  Now the topology on $\Sub(G)$ is the smallest topology such
that the embedding $j$ is continuous.  By definition, the base of this
topology consists of sets of the form
$$
U_G(S^+,S^-)=
\{H\in\Sub(G) \mid S^+\subset H \mbox{ and } S^-\cap H=\emptyset\},
$$
where $S^+$ and $S^-$ run independently over all finite subsets of $G$.
Notice that $U_G(S^+_1,S^-_1)\cap U_G(S^+_2,S^-_2)=
U_G(S^+_1\cup S^+_2,S^-_1\cup S^-_2)$.

The topological space $\Sub(G)$ is ultrametric and compact (since
$\{0,1\}^G$ is ultrametric and compact, and $j(\Sub(G))$ is closed in
$\{0,1\}^G$).  Suppose $g_1,g_2,g_3,\dots$ is a complete list of elements
of the group $G$.  For any subgroups $H_1,H_2\subset G$ let $d(H_1,H_2)=0$
if $H_1=H_2$; otherwise let $d(H_1,H_2)=2^{-n}$, where $n$ is the smallest
index such that $g_n$ belongs to the symmetric difference of $H_1$ and
$H_2$.  Then $d$ is a distance function on $\Sub(G)$ compatible with the
topology.

Note that the above construction also applies to a finite group $G$, in
which case $\Sub(G)$ is a finite set with the discrete topology.

The following three lemmas explore properties of the topological space
$\Sub(G)$.

\begin{lemma}\label{sub1}
The intersection of subgroups is a continuous operation on the space
$\Sub(G)$.
\end{lemma}

\begin{proof}
We have to show that the mapping $I:\Sub(G)\times\Sub(G)\to\Sub(G)$ defined
by $I(H_1,H_2)=H_1\cap H_2$ is continuous.  Take any finite sets
$S^+,S^-\subset G$.  Given subgroups $H_1,H_2\subset G$, the intersection
$H_1\cap H_2$ is an element of the set $U_G(S^+,S^-)$ if and only if
$H_1\in U_G(S^+,S_1)$ and $H_2\in U_G(S^+,S_2)$ for some sets $S_1$ and
$S_2$ such that $S_1\cup S_2=S^-$.  Clearly, the sets $S_1$ and $S_2$ are
finite.  It follows that
$$
I^{-1}\bigl(U_G(S^+,S^-)\bigr)=\bigcup_{S_1,S_2\,:\,S_1\cup S_2=S^-}
U_G(S^+,S_1)\times U_G(S^+,S_2).
$$
It remains to notice that any open subset of $\Sub(G)$ is a union of sets
of the form $U_G(S^+,S^-)$ while any set of the form $U_G(S^+,S_1)\times
U_G(S^+,S_2)$ is open in $\Sub(G)\times\Sub(G)$.
\end{proof}

\begin{lemma}\label{sub2}
For any subgroups $H_1$ and $H_2$ of the group $G$, let $H_1\vee H_2$
denote the subgroup generated by all elements of $H_1$ and $H_2$.  Then
$\vee$ is a Borel measurable operation on $\Sub(G)$.
\end{lemma}

\begin{proof}
We have to show that the mapping $J:\Sub(G)\times\Sub(G)\to\Sub(G)$ defined
by $J(H_1,H_2)=H_1\vee H_2$ is Borel measurable.  Take any $g\in G$ and
consider arbitrary subgroups $H_1,H_2\in\Sub(G)$ such that $J(H_1,H_2)\in
U_G(\{g\},\emptyset)$, i.e., $H_1\vee H_2$ contains $g$.  The element $g$
can be represented as a product $g=h_1h_2\dots h_k$, where each $h_i$
belongs to $H_1$ or $H_2$.  Let $S_1$ denote the set of all elements of
$H_1$ in the sequence $h_1,h_2,\dots,h_k$ and $S_2$ denote the set of all
elements of $H_2$ in the same sequence.  Then the element $g$ belongs to
$K_1\vee K_2$ for any subgroups $K_1\in U_G(S_1,\emptyset)$ and $K_2\in
U_G(S_2,\emptyset)$.  Hence the pair $(H_1,H_2)$ is contained in the
preimage of $U_G(\{g\},\emptyset)$ under the mapping $J$ along with its
open neighborhood $U_G(S_1,\emptyset)\times U_G(S_2,\emptyset)$.  Thus the
preimage $J^{-1}\bigr(U_G(\{g\},\emptyset)\bigr)$ is an open set.  Since
the set $U_G(\emptyset,\{g\})$ is the complement of $U_G(\{g\},\emptyset)$,
its preimage under $J$ is closed.

Given finite sets $S^+,S^-\subset G$, the set $U_G(S^+,S^-)$ is the
intersection of sets $U_G(\{g\},\emptyset)$, $g\in S^+$ and
$U_G(\emptyset,\{h\})$, $h\in S^-$.  By the above
$J^{-1}\bigr(U_G(S^+,S^-)\bigr)$ is a Borel set, the intersection of an
open set with a closed one.  Finally, any open subset of $\Sub(G)$ is the
union of some sets $U_G(S^+,S^-)$.  Moreover, it is a finite or countable
union since there are only countably many sets of the form $U_G(S^+,S^-)$.
It follows that the preimage under $J$ of any open set is a Borel set.
\end{proof}

\begin{lemma}\label{sub3}
Suppose $H$ is a subgroup of $G$.  Then $\Sub(H)$ is a closed subset of
$\Sub(G)$.  Moreover, the intrinsic topology on $\Sub(H)$ coincides with
the topology induced by $\Sub(G)$.
\end{lemma}

\begin{proof}
The intrinsic topology on $\Sub(H)$ is generated by all sets of the form
$U_H(P^+,P^-)$, where $P^+$ and $P^-$ are finite subsets of $H$.  The
topology induced by $\Sub(G)$ is generated by all sets of the form
$U_G(S^+,S^-)\cap\Sub(H)$, where $S^+$ and $S^-$ are finite subsets of $G$.
Clearly, $U_G(S^+,S^-)\cap\Sub(H)=U_H(S^+,S^-\cap H)$ if $S^+\subset H$ and
$U_G(S^+,S^-)\cap\Sub(H)=\emptyset$ otherwise.  It follows that the two
topologies coincide.

For any $g\in G$ the open set $U_G(\emptyset,\{g\})$ is also closed in
$\Sub(G)$ as it is the complement of another open set
$U_G(\{g\},\emptyset)$.  Then the set $\Sub(H)$ is closed in $\Sub(G)$
since it is the intersection of closed sets $U_G(\emptyset,\{g\})$ over all
$g\in G\setminus H$.
\end{proof}

Let $A$ be an action of the group $G$ on a set $M$.  Let us consider the
stabilizer $\St_A(x)$ of a point $x\in M$ under the action (see Section
\ref{act}) as the value of a mapping $\St_A:M\to\Sub(G)$.

\begin{lemma}\label{sub4}
Suppose $A$ is a continuous action of the group $G$ on a Hausdorff
topological space $M$.  Then
\begin{itemize}
\item[(i)]
the mapping $\St_A$ is Borel measurable;
\item[(ii)]
$\St_A$ is continuous at a point $x\in M$ if and only if the stabilizer of
$x$ under the action coincides with its neighborhood stabilizer:
$\St^o_A(x)=\St_A(x)$;
\item[(iii)]
if a sequence of points in $M$ converges to the point $x$ and the sequence
of their stabilizers converges to a subgroup $H$, then $\St^o_A(x)\subset
H\subset\St_A(x)$.
\end{itemize}
\end{lemma}

\begin{proof}
For any $g\in G$ let $\Fix_A(g)$ denote the set of all points in $M$ fixed
by the transformation $A_g$.  Let us show that $\Fix_A(g)$ is a closed set.
Take any point $x\in M$ not in $\Fix_A(g)$.  Since the points $x$ and
$A_g(x)$ are distinct, they have disjoint open neighborhoods $X$ and $Y$,
respectively.  Since $A_g$ is continuous, there exists an open neighborhood
$Z$ of $x$ such that $A_g(Z)\subset Y$.  Then $X\cap Z$ is an open
neighborhood of $x$ and $A_g(X\cap Z)$ is disjoint from $X\cap Z$.  In
particular, $X\cap Z$ is disjoint from $\Fix_A(g)$.

For any finite sets $S^+,S^-\subset G$ the preimage of the open set
$U_G(S^+,S^-)$ under the mapping $\St_A$ is
$$
\bigcap_{g\in S^+}\Fix\nolimits_A(g)\setminus
\bigcup_{h\in S^-}\Fix\nolimits_A(h).
$$
This is a Borel set as $\Fix_A(g)$ is closed for any $g\in G$.  Since sets
of the form $U_g(S^+,S^-)$ constitute a base of the topology on $\Sub(G)$,
the mapping $\St_A$ is Borel measurable.

The mapping $\St_A$ is continuous at a point $x\in M$ if and only if $x$ is
an interior point in the preimage under $\St_A$ of any set $U_G(S^+,S^-)$
containing $\St_A(x)$.  The latter holds true if and only if $x$ is an
interior point in any set $\Fix_A(g)$ containing this point.  Clearly, $x$
is an interior point of $\Fix_A(g)$ if and only if $g$ belongs to the
neighborhood stabilizer $\St_A^o(x)$.  Thus $\St_A$ is continuous at $x$ if
and only if any element of $\St_A(x)$ belongs to $\St_A^o(x)$ as well.

Now suppose that a sequence $x_1,x_2,\dots$ of points in $M$ converges to
the point $x$ and, moreover, the stabilizers $\St_A(x_1),\St_A(x_2),\dots$
converge to a subgroup $H$.  Consider an arbitrary $g\in G$.  In the case
$g\in H$, the subgroup $H$ belongs to the open set $U_G(\{g\},\emptyset)$.
Since $\St_A(x_n)\to H$ as $n\to\infty$, we have $\St_A(x_n)\in
U_G(\{g\},\emptyset)$ for large $n$.  In other words, $x_n\in\Fix_A(g)$ for
large $n$.  Since the set $\Fix_A(g)$ is closed, it contains the limit
point $x$ as well.  That is, $g\in\St_A(x)$.  In the case $g\notin H$, the
subgroup $H$ belongs to the open set $U_G(\emptyset,\{g\})$.  Then
$\St_A(x_n)\in U_G(\emptyset,\{g\})$ for large $n$.  In other words,
$x_n\notin\Fix_A(g)$ for large $n$.  Since $x_n\to x$ as $n\to\infty$, the
action of $g$ fixes no neighborhood of $x$.  That is, $g\notin\St^o_A(x)$.
\end{proof}

The group $G$ acts naturally on the set $\Sub(G)$ by conjugation.  The
action $\cC:G\times\Sub(G)\to\Sub(G)$ is given by $\cC(g,H)=gHg^{-1}$.
This action is continuous.  Indeed, one easily observes that
$\cC_g^{-1}\bigl(U_G(S_1,S_2)\bigr)=U_G(g^{-1}S_1g,g^{-1}S_2g)$ for all
$g\in G$ and finite sets $S_1,S_2\subset G$.

\begin{proposition}\label{sub5}
The action $\cC$ of the group $G$ on $\Sub(G)$ is continuously conjugated
to the action $\cA$ on the space $\Sch(G,S)$ of the marked Schreier graphs
of $G$ relative to a generating set $S$.  Moreover, the mapping
$f:\Sub(G)\to\Sch(G,S)$ given by $f(H)=\Ga_{\coset}^*(G,S;H)$ is a
continuous conjugacy.
\end{proposition}

\begin{proof}
Proposition \ref{sch4} implies that the mapping $f$ is bijective.

Consider arbitrary element $g$ and subgroup $H$ of the group $G$.  The
stabilizer of the coset $gH$ under the action $\adj_{G,H}$ consists of
those $g_0\in G$ for which $g_0gH=gH$.  The latter condition is equivalent
to $g^{-1}g_0g\in H$.  Therefore the stabilizer is $gHg^{-1}=\cC_g(H)$.
As $\cA_g\bigl(\Ga_{\coset}^*(G,S;H)\bigr)
=\Ga_{\Sch}^*(G,S;\adj_{G,H},gH)$, it follows from Proposition \ref{sch4}
that $\cA_g(f(H))=f(\cC_g(H))$.  Thus $f$ conjugates the action $\cC$ with
$\cA$.

Now we are going to show that for any finite sets $S^+,S^-\subset G$ the
image of the open set $U_G(S^+,S^-)$ under the mapping $f$ is open in
$\Sch(G,S)$.  Let $\Ga_{\Sch}^*(G,S;A,x)$ be an arbitrary graph in that
image.  Any element $g\in G$ can be represented as a product $s_1s_2\dots
s_k$, where $s_i\in S$.  Let us fix such a representation and denote by
$\ga_g$ the unique directed path in $\Ga_{\Sch}^*(G,S;A,x)$ with beginning
$x$ and code word $s_k\dots s_2s_1$.  Then the end of the path $\ga_g$ is
$A_g(x)$.  In particular, the path $\ga_g$ is closed if and only if
$g\in\St_A(x)$.  By Proposition \ref{sch4}, the preimage of the graph
$\Ga_{\Sch}^*(G,S;A,x)$ under $f$ is $\St_A(x)$.  Since $\St_A(x)\in
U_G(S^+,S^-)$, the path $\ga_g$ is closed for $g\in S^+$ and not closed for
$g\in S^-$.  Let $\Ga_0$ denote the smallest subgraph of
$\Ga_{\Sch}^*(G,S;A,x)$ containing all paths $\ga_g$, $g\in S^+\cup S^-$.
Clearly, $\Ga_0$ is a marked graph, finite and connected.  Hence
$\Ga_0\in\MG_0$.  For any marked Schreier graph $\Ga_{\Sch}^*(G,S;B,y)$ in
$\cU(\Ga_0,\emptyset)$, the directed path with beginning $y$ and the same
code word as in $\ga_g$ is closed for all $g\in S^+$ and not closed for all
$g\in S^-$.  It follows that $\St_B(y)\in U_G(S^+,S^-)$.  Therefore the
graph $\Ga_{\Sch}^*(G,S;A,x)$ is contained in $f(U_G(S^+,S^-))$ along with
its neighborhood $\cU(\Ga_0,\emptyset)\cap\Sch(G,S)$.

Any open set in $\Sub(G)$ is the union of some sets $U_G(S^+,S^-)$.  Hence
it follows from the above that the mapping $f$ maps open sets onto open
sets.  In other words, the inverse mapping $f^{-1}$ is continuous.  Since
the topological spaces $\Sub(G)$ and $\Sch(G,S)$ are compact, $f$ is
continuous as well.
\end{proof}

Proposition \ref{sub5} allows for a short (although not constructive) proof
of the following statement.

\begin{proposition}\label{sub6}
Any subgroup of finite index of a finitely generated group is also finitely
generated.
\end{proposition}

\begin{proof}
Suppose $G$ is a finitely generated group and $H$ is a subgroup of $G$ of
finite index.  Let $S$ be a finite symmetric generating set for $G$.  By
Proposition \ref{sub5}, the space $\Sub(G)$ of subgroups of $G$ is
homeomorphic to the space $\Sch(G,S)$ of marked Schreier graphs of $G$
relative to the generating set $S$.  Moreover, there is a homeomorphism
that maps the subgroup $H$ to the marked Schreier coset graph
$\Ga_{\coset}^*(G,S;H)$.  The vertices of the graph are cosets of $H$ in
$G$.  Since $H$ has finite index in $G$, the graph $\Ga_{\coset}^*(G,S;H)$
is finite.  Notice that any finite graph in the topological space $\MG_0$,
which contains $\Sch(G,S)$, is an isolated point.  It follows that $H$ is
an isolated point in $\Sub(G)$.  Then there exist finite sets
$S^+,S^-\subset G$ such that $H$ is the only element of the open set
$U_G(S^+,S^-)$.  Let $H_0$ be the subgroup of $G$ generated by the finite
set $S^+$.  Since $S^+\subset H$ and $S^-\cap H=\emptyset$, the subgroup
$H_0$ is disjoint from $S^-$.  Thus $H_0\in U_G(S^+,S^-)$ so that $H_0=H$.
\end{proof}

\section{Automorphisms of regular rooted trees}\label{tree}

Consider an arbitrary graph $\Ga$.  Let $\ga$ be a path in this graph,
$v_0,v_1,\dots,v_m$ be consecutive vertices of $\ga$, and $e_1,\dots,e_m$
be consecutive edges.  A {\em backtracking\/} in the path $\ga$ occurs if
$e_{i+1}=e_i$ for some $i$ (then $v_{i+1}=v_{i-1}$).  The graph $\Ga$ is
called a {\em tree\/} if it is connected and admits no closed path of
positive length without backtracking.  In particular, this means no loops
and no multiple edges.  A {\em rooted tree\/} is a tree with a
distinguished vertex called the {\em root}.  Clearly, the root is a synonym
for the marked vertex.  For any integer $n\ge0$ the {\em level\/} $n$ (or
the $n$th level) of the tree is defined as the set of vertices at distance
$n$ from the root.  If $n\ge1$ then any vertex on the $n$th level is joined
to exactly one vertex on the level $n-1$ and, optionally, to some vertices
on the level $n+1$.  The rooted tree is called {\em $k$-regular\/} if every
vertex on any level $n$ is joined to exactly $k$ vertices on level $n+1$.
The $2$-regular rooted tree is also called {\em binary}.

All $k$-regular rooted trees are isomorphic to each other.  A standard
model of such a tree is built as follows.  Let $X$ be a set of cardinality
$k$ referred to as the {\em alphabet\/} (usually $X=\{0,1,\dots,k-1\}$).  A
{\em word\/} (or {\em finite word\/}) in the alphabet $X$ is a finite
string of elements from $X$ (referred to as {\em letters\/}).  The set of
all words in the alphabet $X$ is denoted $X^*$.  $X^*$ is a monoid with
respect to the concatenation (the unit element is the empty word, denoted
$\varnothing$).  Moreover, it is the free monoid generated by elements of
$X$.  Now we define a plain graph $\cT$ with the vertex set $X^*$ in which
two vertices $w_1$ and $w_2$ are joined by an edge if $w_1=w_2x$ or
$w_2=w_1x$ for some $x\in X$.  Then $\cT$ is a $k$-regular rooted tree with
the root $\varnothing$.  The $n$th level of the tree $\cT$ consists of all
words of length $n$.

A bijection $f:X^*\to X^*$ is an automorphism of the rooted tree $\cT$ if
and only if it preserves the length of any word and the length of the
common beginning of any two words.  Given an automorphism $f$ and a word
$u\in X^*$, there exists a unique transformation $h:X^*\to X^*$ such that
$f(uw)=f(u)h(w)$ for all $w\in X^*$.  It is easy to see that $h$ is also an
automorphism of the tree $\cT$.  This automorphism is called the {\em
section\/} of $f$ at the word $u$ and denoted $f|_u$.  A set of
automorphisms of the tree $\cT$ is called {\em self-similar\/} if it is
closed under taking sections.  For any automorphisms $f$ and $h$ and any
word $u\in X^*$ one has $(fh)|_u=f|_{h(u)}h|_u$ and $f^{-1}|_u
=\bigl(f|_{f^{-1}(u)}\bigr)^{-1}$.  It follows that any group of
automorphisms generated by a self-similar set is itself self-similar.

Suppose $G$ is a group of automorphisms of the tree $\cT$.  Let $\al$
denote the natural action of $G$ on the vertex set $X^*$.  Given a word
$u\in X^*$, the section mapping $g\mapsto g|_u$ is a homomorphism when
restricted to the stabilizer $\St_\al(u)$.  If $G$ is self-similar then
this is a homomorphism to $G$.  The self-similar group $G$ is called {\em
self-replicating\/} if for any $u\in X^*$ the mapping $g\mapsto g|_u$ maps
the subgroup $\St_\al(u)$ onto the entire group $G$.

Suppose that letters of the alphabet $X$ are canonically ordered:
$x_1,x_2,\dots,x_k$.  For any permutation $\pi$ on $X$ and automorphisms
$h_1,h_2,\dots,h_k$ of the tree $\cT$ we denote by $\pi(h_1,h_2,\dots,h_k)$
a transformation $f:X^*\to X^*$ given by $f(x_iw)=\pi(x_i)h_i(w)$ for all
$w\in X^*$ and $1\le i\le k$.  It is easy to observe that $f$ is also an
automorphism of $\cT$ and $h_i=f|_{x_i}$ for $1\le i\le k$.  The expression
$\pi(h_1,h_2,\dots,h_k)$ is called the {\em wreath recursion\/} for $f$.
Any self-similar set of automorphisms $f_j$, $j\in J$ satisfies a system of
``self-similar'' wreath recursions
$$
f_j=\pi_j(f_{m(j,x_1)},f_{m(j,x_2)},\dots,f_{m(j,x_k)}), \ j\in J,
$$
where $\pi_j$, $j\in J$ are permutations on $X$ and $m$ maps $J\times X$ to
$J$.

\begin{lemma}\label{tree0}
Any system of self-similar wreath recursions over the alphabet $X$ is
satisfied by a unique self-similar set of automorphisms of the regular
rooted tree $\cT$.
\end{lemma}

\begin{proof}
Consider a system of wreath recursions $f_j=\pi_j(f_{m(j,x_1)},\dots,
f_{m(j,x_k)})$, $j\in J$, where $\pi_j$, $j\in J$ are permutations on $X$
and $m$ is a mapping of $J\times X$ to $J$.  We define transformations
$F_j$, $j\in J$ of the set $X^*$ inductively as follows.  First
$F_j(\varnothing)=\varnothing$ for all $j\in J$.  Then, once the
transformations are defined on words of a particular length $n\ge0$, we let
$F_j(x_iw)=\pi_j(x_i)F_{m(j,x_i)}(w)$ for all $j\in J$, $1\le i\le k$, and
words $w$ of length $n$.  By definition, each $F_j$ preserves the length of
words.  Besides, it follows by induction on $n$ that $F_j$ is bijective
when restricted to words of length $n$ and that $F_j$ preserves having a
common beginning of length $n$ for any two words.  Therefore each $F_j$ is
an automorphism of the tree $\cT$.  By construction, the automorphisms
$F_j$, $j\in J$ form a self-similar set satisfying the above system of
wreath recursions.  Moreover, they provide the only solution to that
system.
\end{proof}

An {\em infinite path\/} in the tree $\cT$ is an infinite sequence of
vertices $v_0,v_1,v_2,\dots$ together with a sequence of edges
$e_1,e_2\dots$ such that the endpoints of any $e_i$ are $v_{i-1}$ and
$v_i$.  The vertex $v_0$ is the beginning of the path.  Clearly, the path
is uniquely determined by the sequence of vertices alone.  The {\em
boundary\/} of the rooted tree $\cT$, denoted $\dT$, is the set of all
infinite paths without backtracking that begin at the root.  There is a
natural one-to-one correspondence between $\dT$ and the set $X^{\bN}$ of
infinite words over the alphabet $X$.  Namely, an infinite word
$x_1x_2x_3\dots$ corresponds to the path going through the vertices
$\varnothing,x_1,x_1x_2,x_1x_2x_3,\dots$.  The set $X^{\bN}$ is equipped
with the product topology and the uniform Bernoulli measure.  This allows
us to regard the tree boundary $\dT$ as a compact topological space with a
Borel probability measure (called {\em uniform\/}).

Suppose $G$ is a group of automorphisms of the regular rooted tree $\cT$.
The natural action of $G$ on the vertex set $X^*$ gives rise to an action
on the boundary $\dT$.  The latter is continuous and preserves the uniform
measure on $\dT$.

\begin{proposition}[\cite{G2}]\label{tree1}
Let $G$ be a countable group of automorphisms of a regular rooted tree
$\cT$.  Then the following conditions are equivalent:
\begin{itemize}
\item[(i)]
the group $G$ acts transitively on each level of the tree;
\item[(ii)]
the action of $G$ on the boundary $\dT$ of the tree is topologically
transitive;
\item[(iii)]
the action of $G$ on $\dT$ is minimal;
\item[(iv)]
the action of $G$ on $\dT$ is ergodic with respect to the uniform measure;
\item[(v)]
the action of $G$ on $\dT$ is uniquely ergodic.
\end{itemize}
\end{proposition}

Let $G$ be a countable group of automorphisms of a regular rooted tree
$\cT$.  Let $\al$ denote the natural action of $G$ on the vertex set of the
tree $\cT$ and $\be$ denote the induced action of $G$ on the boundary $\dT$
of the tree.

\begin{proposition}\label{tree2}
The mapping $\St_\be$ is continuous on a residual (dense $G_\delta$) set.
\end{proposition}

\begin{proof}
For any $g\in G$ let $\Fix_\be(g)$ denote the set of all points in $\dT$
fixed by the transformation $\be_g$.  If $g\in\St_\be(\xi)$ but
$g\notin\St^o_\be(\xi)$, then $\xi$ is a boundary point of the set
$\Fix_\be(g)$, and vice versa.  Since $\Fix_\be(g)$ is a closed set, its
boundary is a closed, nowhere dense set.  It follows that the set of points
$\xi\in\dT$ such that $\St^o_\be(\xi)=\St_\be(\xi)$ is the intersection of
countably many dense open sets (it is dense since $\dT$ is a complete
metric space).  By Lemma \ref{sub4}, the latter set consists of points at
which the mapping $\St_\be$ is continuous.
\end{proof}

The mapping $\St_\be$ is Borel due to Lemma \ref{sub4}.  If $\St_\be$ is
injective then, according to the descriptive set theory, it also maps Borel
sets onto Borel sets (see, e.g., \cite{K}).  The following two lemmas show
the same can hold under a little weaker condition.

\begin{lemma}\label{tree3}
Assume that for any points $\xi,\eta\in\dT$ either $\St_\be(\xi)
=\St_\be(\eta)$ or $\St_\be(\xi)$ is not contained in $\St_\be(\eta)$.
Then the mapping $\St_\be$ maps any open set, any closed set, and any
intersection of an open set with a closed one onto Borel sets.
\end{lemma}

\begin{proof}
First let us show that $\St_\be$ maps any closed subset $C$ of the boundary
$\dT$ onto a Borel subset of $\Sub(G)$.  For any positive integer $n$ let
$C_n$ denote the set of all words of length $n$ in the alphabet $X$ that
are beginnings of infinite words in $C$.  Further, let $W_n$ be the union
of sets $\Sub(\St_\al(w))$ over all words $w\in C_n$.  Finally, let $W$ be
the intersection of the sets $W_n$ over all $n\ge1$.  By Lemma \ref{sub3},
the set $\Sub(H)$ is closed in $\Sub(G)$ for any subgroup $H\in\Sub(G)$.
Hence each $W_n$ is closed as the union of finitely many closed sets.  Then
the intersection $W$ is closed as well.

The stabilizer $\St_\be(\xi)$ of an infinite word $\xi\in\dT$ is a subgroup
of the stabilizer $\St_\al(w)$ of a finite word $w\in X^*$ whenever $w$ is
a beginning of $\xi$.  It follows that $\St_\be(\xi)\in W$ for all $\xi\in
C$.  By construction of the set $W$, any subgroup of an element of $W$ is
also an element of $W$.  Hence $W$ contains all subgroups of the groups
$\St_\be(\xi)$, $\xi\in C$.

Conversely, for any subgroup $H\in W$ there is a sequence of words
$w_1,w_2,\dots$ such that $w_n\in C_n$ and $H\subset\St_\al(w_n)$ for
$n=1,2,\dots$.  Since the number of words of a fixed length is finite, one
can find nested infinite sets of indices $I_1\supset I_2\supset\dots$ such
that the beginning of length $k$ of the word $w_n$ is the same for all
$n\in I_k$.  Choose an increasing sequence of indices $n_1,n_2,\dots$ such
that $n_k\in I_k$ for all $k$, and let $w'_k$ be the beginning of length
$k$ of the word $w_{n_k}$.  Then $w'_k\in C_k$ as $w_{n_k}\in C_{n_k}$.
Besides, $\St_\al(w_{n_k})\subset\St_\al(w'_k)$, in particular, the group
$H$ is a subgroup of $\St_\al(w'_k)$.  By construction, the word $w'_k$ is
a beginning of $w'_m$ whenever $k<m$.  Therefore all $w'_k$ are beginnings
of the same infinite word $\xi'\in\dT$.  Since every beginning of $\xi'$
coincides with a beginning of some infinite word in $C$ and the set $C$ is
closed, it follows that $\xi'\in C$.  The stabilizer $\St_\be(\xi')$ is the
intersection of stabilizers $\St_\al(w'_k)$ over all $k\ge1$.  Hence $H$ is
a subgroup of $\St_\be(\xi')$.

By the above a subgroup $H$ of $G$ belongs to the set $W$ if and only if it
is a subgroup of the stabilizer $\St_\be(\xi)$ for some $\xi\in C$.  The
assumption of the lemma implies that stabilizers $\St_\be(\xi)$, $\xi\in C$
can be distinguished as the maximal subgroups in the set $W$.  That is,
such a stabilizer is an element of $W$ which is not a proper subgroup of
another element of $W$.  For any $g\in G$ we define a transformation
$\psi_g$ of $\Sub(G)$ by $\psi_g(H)=\langle g\rangle\vee H$, where $\langle
g\rangle$ is a cyclic subgroup of $G$ generated by $g$.  The group
$\psi_g(H)$ is generated by $g$ and all elements of the group $H$.
Clearly, a subgroup $H\in W$ is not maximal in $W$ if and only if
$\psi_g(H)\in W$ for some $g\notin H$.  An equivalent condition is that $H$
belongs to the set $W'_g=W\cap\psi_g^{-1}(W)\cap U_G(\emptyset,\{g\})$.  It
follows from Lemma \ref{sub2} that the mapping $\psi_g$ is Borel
measurable.  Therefore $W'_g$ is a Borel set.  Now the image of the set $C$
under the mapping $\St_\be$ is the difference of the closed set $W$ and the
union of Borel sets $W'_g$, $g\in G$.  Hence this image is a Borel set.

Any open set $D\subset\dT$ is the union of a finite or countable collection
of cylinders $Z_1,Z_2,\dots$, which are both open and closed sets.  By the
above each cylinder is mapped by $\St_\be$ onto a Borel set in $\Sub(G)$.
Then the union $D$ is mapped onto the union of images of the cylinders,
which is a Borel set as well.  Further, for any closed set $C\subset\dT$
the intersection $C\cap D$ is the union of closed sets $C\cap Z_1,
C\cap Z_2,\dots$.  Hence it is also mapped by $\St_\be$ onto a Borel set.
\end{proof}

\begin{lemma}\label{tree4}
Under the assumption of Lemma \ref{tree3}, if the mapping $\St_\be$ is
finite-to-one, i.e., the preimage of any subgroup in $\Sub(G)$ is finite,
then it maps Borel sets onto Borel sets.
\end{lemma}

\begin{proof}
Recall that the class $\fB$ of the Borel sets in $\dT$ is the smallest
collection of subsets of $\dT$ that contains all closed sets and is closed
under taking countable intersections, countable unions, and complements.
Let $\fU$ denote the smallest collection of subsets of $\dT$ that contains
all closed sets and is closed under taking countable intersections of
nested sets and countable unions of any sets.  Note that $\fU$ is well
defined; it is the intersection of all collections satisfying these
conditions.  In particular, $\fU\subset\fB$.  Further, let $\fW$ denote the
collection of all Borel sets in $\dT$ mapped onto Borel sets in $\Sub(G)$
by the mapping $\St_\be$.

For any mapping $f:\dT\to\Sub(G)$ and any sequence $U_1,U_2,\dots$ of
subsets of $\dT$ the image of the union $U_1\cup U_2\cup\dots$ under $f$ is
the union of images $f(U_1),f(U_2),\dots$.  On the other hand, the image of
the intersection $U_1\cap U_2\cap\dots$ under $f$ is contained in
$f(U_1)\cap f(U_2)\cap\dots$ but need not coincide with the latter when the
mapping $f$ is not one-to-one.  The two sets do coincide if $f$ is
finite-to-one and $U_1\supset U_2\supset\dots$.    Since the mapping
$\St_\be$ is assumed to be finite-to-one, it follows that the collection
$\fW$ is closed under taking countable intersections of nested sets and
countable unions of any sets.  By Lemma \ref{tree3}, $\fW$ contains all
closed sets.  Therefore $\fU\subset\fW$.

To complete the proof, we are going to show that $\fU=\fB$, which will
imply that $\fW=\fB$.  Given a set $Y\in\fU$, let $\fU_Y$ denote the
collection of all sets $U\in\fU$ such that the intersection $U\cap Y$ also
belongs to $\fU$.  For any sequence $U_1,U_2,\dots$ of elements of $\fU_Y$
we have
$$
\left(\bigcup\nolimits_{n\ge1}U_n\right)\cap Y
=\bigcup\nolimits_{n\ge1}(U_n\cap Y),
\qquad
\left(\bigcap\nolimits_{n\ge1}U_n\right)\cap Y
=\bigcap\nolimits_{n\ge1}(U_n\cap Y).
$$
Besides, the sets $U_1\cap Y,U_2\cap Y,\dots$ are nested whenever the sets
$U_1,U_2,\dots$ are nested.  It follows that the class $\fU_Y$ is closed
under taking countable intersections of nested sets and countable unions of
any sets.  Consequently, $\fU_Y=\fU$ whenever $\fU_Y$ contains all closed
sets.  The latter condition obviously holds if the set $Y$ is itself
closed.  Notice that for any sets $Y,Z\in\fU$ we have $Z\in\fU_Y$ if and
only if $Y\in\fU_Z$.  Since $\fU_Y=\fU$ for any closed set $Y$, it follows
that $\fU_Z$ contains all closed sets for any $Z\in\fU$.  Then $\fU_Z=\fU$
for any $Z\in\fU$.  In other words, the class $\fU$ is closed under taking
finite intersections.  Combining finite intersections with countable
intersections of nested sets, we can obtain any countable intersection of
sets from $\fU$.  Namely, if $U_1,U_2,\dots$ are arbitrary elements of
$\fU$, then their intersection coincides with the intersection of sets
$Y_n=U_1\cap U_2\cap\dots\cap U_n$, $n=1,2,\dots$, which are nested:
$Y_1\supset Y_2\supset\dots$.  Therefore $\fU$ is closed under taking any
countable intersections.

Let $\fU'$ be the collection of complements in $\dT$ of all sets from
$\fU$.  For any subsets $U_1,U_2,\dots$ of $\dT$ the complement of their
union is the intersection of their complements $\dT\setminus U_1,
\dT\setminus U_2,\dots$ while the complement of their intersection is the
union of their complements.  Since the class $\fU$ is closed under taking
countable intersections and countable unions, so is $\fU'$.  Further, any
open subset of $\dT$ is the union of at most countably many cylinders,
which are closed (as well as open) sets.  Therefore $\fU$ contains all open
sets.  Then $\fU'$ contains all closed sets.  Now it follows that
$\fU\subset\fU'$.  In other words, the class $\fU$ is closed under taking
complements.

Thus the collection $\fU$ is closed under taking any countable
intersections and complements.  This implies that $\fU=\fB$.
\end{proof}

Let $A$ be a continuous action of a countable group $G$ on a compact metric
space $M$.  Let $\Om$ denote the image of $M$ under the mapping $\St_A$.

\begin{lemma}\label{tree5}
Assume that for any distinct points $x,y\in M$ the neighborhood stabilizer
$\St^o_A(x)$ is not contained in the stabilizer $\St_A(y)$.  Then the
inverse of $\St_A$, defined on the set $\Om$, can be extended to a
continuous mapping of the closure of $\Om$ onto $M$.
\end{lemma}

\begin{proof}
Since $\St^o_A(x)$ is a subgroup of $\St_A(x)$ for any $x\in M$, the
assumption of the lemma implies that the mapping $\St_A$ is one-to-one so
that the inverse is well defined on $\Om$.  To prove that the inverse can
be extended to a continuous mapping of the closure of $\Om$ onto $M$, it is
enough to show that any sequence $x_1,x_2,\dots$ of points in $M$ is
convergent whenever the sequence of stabilizers $\St_A(x_1),\St_A(x_2),
\dots$ converges in $\Sub(G)$.  Suppose that $\St_A(x_n)\to H$ as
$n\to\infty$.  Since $M$ is a compact metric space, the sequence
$x_1,x_2,\dots$ has at least one limit point.  By Lemma \ref{sub4}(iii),
any limit point $x$ satisfies $\St^o_A(x)\subset H\subset\St_A(x)$.  In
particular, $\St^o_A(x)\subset\St_A(y)$ for any limit points $x$ and $y$.
Then $x=y$ due to the assumption of the lemma.  It follows that the
sequence $x_1,x_2,\dots$ is convergent.
\end{proof}

\section{The Grigorchuk group}\label{grig}

Let $X=\{0,1\}$ be the binary alphabet, $X^*$ be the set of finite words
over $X$ regarded as the vertex set of a binary rooted tree $\cT$, and
$X^{\bN}$ be the set of infinite words over $X$ regarded as the boundary
$\dT$ of the tree $\cT$.

We define the Grigorchuk group $\cG$ as a self-similar group of
automorphisms of the tree $\cT$ (for alternative definitions, see
\cite{G1}).  The group is generated by four automorphisms $a,b,c,d$ that,
together with the trivial automorphism, form a self-similar set.  Consider
the following system of wreath recursions:
$$
\left\{
\begin{array}{l}
a=(0\,1)(e,e),\\
b=(a,c),\\
c=(a,d),\\
d=(e,b),\\
e=(e,e).
\end{array}
\right.
$$
By Lemma \ref{tree0}, this system uniquely defines a self-similar set of
automorphisms of the tree $\cT$.  The automorphism $e$ is clearly the
identity (e.g., by Lemma \ref{tree0}).  It is the unity of the group $\cG$.
We shall denote the unity by $\1$ to avoid confusion with a letter of the
alphabet $X$.  The set $S=\{a,b,c,d\}$ shall be considered the standard set
of generators for the group $\cG$.

All $4$ generators of the Grigorchuk group are involutions.  Indeed, the
transformations $a^2,b^2,c^2,d^2,\1$ form a self-similar set satisfying
wreath recursions $a^2=(\1,\1)$, $b^2=(a^2,c^2)$, $c^2=(a^2,d^2)$,
$d^2=(\1,b^2)$, and $\1=(\1,\1)$.  Then Lemma \ref{tree0} implies that
$a^2=b^2=c^2=d^2=\1$.  This fact allows us to regard the Schreier graphs of
the group $\cG$ relative to the generating set $S$ as graphs with
undirected edges (as explained in Section \ref{graph}).

Since $a^2=\1$, the automorphisms $bcd$, $cdb$, $dbc$, and $\1$ form a
self-similar set satisfying wreath recursions $bcd=(\1,cdb)$,
$cdb=(\1,dbc)$, $dbc=(\1,bcd)$, and $\1=(\1,\1)$.  Lemma \ref{tree0}
implies that $bcd=cdb=dbc=\1$.  Then $bc=bcd^2=d=d^2bc=bc$.  It follows
that $\{\1,b,c,d\}$ is a subgroup of $\cG$ isomorphic to the Klein
$4$-group.

We denote by $\al$ the generic action of the group $\cG$ on vertices of the
binary rooted tree $\cT$.  The induced action on the boundary $\dT$ of the
tree is denoted $\be$.  For brevity, we write $g(\xi)$ instead of
$\be_g(\xi)$.  The action of the generator $a$ is very simple: it changes
the first letter in every finite or infinite word while keeping the other
letters intact.  In particular, the empty word is the only word fixed by
$a$.  To describe the action of the other generators, we need three
observations.  First of all, $b$, $c$, and $d$ fix one-letter words.
Secondly, any word beginning with $0$ is fixed by $d$ while $b$ and $c$
change only the second letter in such a word.  Thirdly, the section mapping
$g\mapsto g|_1$ induces a cyclic permutation on the set $\{b,c,d\}$.  It
follows that a finite or infinite word $w$ is simultaneouly fixed by $b$,
$c$, and $d$ if it contains no zeros or the only zero is the last letter.
Otherwise two of the three generators change the letter following the first
zero in $w$ (keeping the other letters intact) while the third generator
fixes $w$.  In the latter case, it is the position $k$ of the first zero in
$w$ that determines the generator fixing $w$.  Namely, $b(w)=w$ if
$k\equiv0\bmod3$, $c(w)=w$ if $k\equiv2\bmod3$, and $d(w)=w$ if
$k\equiv1\bmod3$.

\begin{lemma}\label{grig1}
The group $\cG$ is self-replicating.
\end{lemma}

\begin{proof}
We have to show that for any word $w\in X^*$ the section mapping $g\mapsto
g|_w$ maps the stabilizer $\St_\al(w)$ onto the entire group $\cG$.  Let
$W$ be the set of all words with this property.  Clearly, $\varnothing\in
W$ as $\St_\al(\varnothing)=\cG$ and $g|_\varnothing=g$ for all $g\in\cG$.
Suppose $w_1,w_2\in W$.  Given an arbitrary $g\in\cG$, there exists
$g'\in\cG$ such that $g'(w_2)=w_2$ and $g'|_{w_2}=g$.  Further, there
exists $g''\in\cG$ such that $g''(w_1)=w_1$ and $g''|_{w_1}=g'$.  Then
$g''(w_1w_2)=g''(w_1)g''|_{w_1}(w_2)=w_1w_2$ and $g''|_{w_1w_2}
=(g''|_{w_1})|_{w_2}=g$.  Since $g$ is arbitrary, $w_1w_2\in W$.  That is,
the set $W$ is closed under concatenation.

Any automorphism of the tree $\cT$ either interchanges the vertices $0$ and
$1$ or fixes them both.  Hence the stabilizer $\St_\al(0)$ coincides with
$\St_\al(1)$.  This stabilizer contains the elements $b,c,d,aba,aca,ada$.
The wreath recursions for these elements are $b=(a,c)$, $c=(a,d)$,
$d=(\1,b)$, $aba=(c,a)$, $aca=(d,a)$, $ada=(b,\1)$.  It follows that the
images of the group $\St_\al(0)$ under the section mappings $g\mapsto g|_0$
and $g\mapsto g|_1$ contain the generating set $S$.  As the restrictions of
these mappings to $\St_\al(0)$ are homomorphisms, both images coincide with
$\cG$.  Therefore the words $0$ and $1$ are in the set $W$.  By the above
$W$ is closed under concatenation and contains the empty word.  This
implies $W=X^*$.
\end{proof}

The orbits of the actions $\al$ and $\be$ are very easy to describe.

\begin{lemma}\label{grig2}
The group $\cG$ acts transitively on each level of the binary rooted tree
$\cT$.  Any two infinite words in $\dT$ are in the same orbit of the action
$\be$ if and only if they differ in only finitely many letters.
\end{lemma}

\begin{proof}
For any infinite word $\xi\in\dT$ and any generator $h\in\{a,b,c,d\}$ the
infinite word $h(\xi)$ differs from $\xi$ in at most one letter.  Any
$g\in\cG$ can be represented as a product $g=h_1h_2\dots h_k$, where each
$h_i$ is in $\{a,b,c,d\}$.  It follows that for any $\xi\in\dT$ the
infinite words $g(\xi)$ and $\xi$ differ in at most $k$ letters.  Thus any
two infinite words in the same orbit of the action $\be$ differ in only
finitely many letters.

Now we are going to show that for any finite words $w_1,w_2\in X^*$ of the
same length there exists $g\in\cG$ such that $g(w_1)=w_2$ and
$g|_{w_1}=\1$.  Equivalently, $g(w_1\xi)=w_2\xi$ for all $\xi\in\dT$.  This
will complete the proof of the lemma.  Indeed, the claim contains the
statement that the group $\cG$ acts transitively on each level of the tree
$\cT$.  Moreover, it implies that two infinite words in $\dT$ are in the
same orbit of the action $\be$ whenever they differ in a finite number of
letters.

We prove the claim by induction on the length $n$ of the words $w_1$ and
$w_2$.  The case $n=0$ is trivial.  Here $w_1$ and $w_2$ are the empty
words so that we take $g=\1$.  Now assume that the claim is true for all
pairs of words of specific length $n\ge0$ and consider words $w_1$ and
$w_2$ of length $n+1$.  Let $x_1$ be the first letter of $w_1$ and $x_2$ be
the first letter of $w_2$.  Then $w_1=x_1u_1$ and $w_2=x_2u_2$, where $u_1$
and $u_2$ are words of length $n$.  By the inductive assumption, there
exists $h\in\cG$ such that $h(u_1\xi)=u_2\xi$ for all $\xi\in\dT$.  Since
the group $\cG$ is self-replicating, there exists $g_0\in\cG$ such that
$g_0(x_1\eta)=x_1h(\eta)$ for all $\eta\in\dT$.  In particular,
$g_0(x_1u_1\xi)=x_1u_2\xi$ for all $\xi\in\dT$.  It remains to take $g=g_0$
if $x_2=x_1$ and $g=ag_0$ otherwise.  Then $g(x_1u_1\xi)=x_2u_2\xi$ for all
$\xi\in\dT$.
\end{proof}

\begin{lemma}\label{grig3}
Suppose $w_1$ and $w_2$ are words in the alphabet $\{0,1\}$ such that $w_1$
is not a beginning of $w_2$ while $w_2$, even with the last two letters
deleted, is not a beginning of $w_1$.  Then there exists $g\in\cG$ that
does not fix $w_2$ while fixing all words with beginning $w_1$.
\end{lemma}

\begin{proof}
First we consider a special case when $w_2=100$.  To satisfy the assumption
of the lemma, the word $w_1$ has to begin with $0$.  Then we can take
$g=d$.  Indeed, the transformation $d$ fixes all words that begin with $0$,
which includes all words with beginning $w_1$.  At the same time,
$d(100)=1b(00)=10a(0)=101\ne100$.

Next we consider a slightly more general case when $w_2$ is an arbitrary
word of length $3$.  By Lemma \ref{grig2}, the group $\cG$ acts
transitively on the third level of the tree $\cT$.  Therefore $h(w_2)=100$
for some $h\in\cG$.  The words $h(w_1)$ and $h(w_2)$ satisfy the assumption
of the lemma since the words $w_1$ and $w_2$ do.  By the above, $dh(w_2)\ne
h(w_2)$ while $d(h(w_1)u)=h(w_1)u$ for all $u\in X^*$.  Let $g=h^{-1}dh$.
Then $g(w_2)\ne w_2$ while $g(w_1w)=w_1w$ for all $w\in X^*$.

Finally, consider the general case.  Let $w_0$ be the longest common
beginning of the words $w_1$ and $w_2$.  Then $w_1=w_0u_1$ and
$w_2=w_0u_2$, where the words $u_1$ and $u_2$ also satisfy the assumption
of the lemma.  In particular, $u_1$ is nonempty and the length of $u_2$ is
at least $3$.  We have $u_2=u'_2u''_2$, where $u'_2,u''_2\in X^*$ and the
length of $u'_2$ is $3$.  Since the first letters of the words $u_1$ and
$u'_2$ are distinct, these words satisfy the assumption of the lemma.  By
the above there exists $g_0\in\cG$ such that $g_0(u'_2)\ne u'_2$ and
$g_0(u_1u)=u_1u$ for all $u\in X^*$.  Since the group $\cG$ is
self-replicating, there exists $g\in\cG$ such that $g(w_0w)=w_0g_0(w)$ for
all $w\in X^*$.  Then $g$ does not fix the word $w_0u'_2$ while fixing all
words with beginning $w_1$.  Since $w_0u'_2$ is a beginning of $w_2$, the
transformation $g$ does not fix $w_2$ as well.
\end{proof}

\begin{lemma}\label{grig4}
For any distinct points $\xi,\eta\in\dT$ the neighborhood stabilizer
$\St^o_\be(\xi)$ is not contained in $\St_\be(\eta)$.
\end{lemma}

\begin{proof}
Let $n$ denote the length of the longest common beginning of the distinct
infinite words $\xi$ and $\eta$.  Let $w_1$ be the beginning of $\xi$ of
length $n+1$ and $w_2$ be the beginning of $\eta$ of length $n+3$.  It is
easy to see that the words $w_1$ and $w_2$ satisfy the assumption of Lemma
\ref{grig3}.  Therefore there exists a transformation $g\in\cG$ that does
not fix $w_2$ while fixing all finite words with beginning $w_1$.  Clearly,
the action of $g$ on $\dT$ fixes all infinite words with beginning $w_1$.
As such infinite words form an open neighborhood of the point $\xi$, we
have $g\in\St^o_\be(\xi)$.  At the same time, $g$ does not fix the infinite
word $\eta$ since it does not fix its beginning $w_2$.  Hence
$g\notin\St_\be(\eta)$ so that $\St^o_\be(\xi)\not\subset\St_\be(\eta)$.
\end{proof}

\begin{lemma}\label{grig5}
$\St^o_\be(\xi)=\St_\be(\xi)$ for any infinite word $\xi\in\dT$ containing
infinitely many zeros.
\end{lemma}

\begin{proof}
We are going to show that, given an automorphism $g\in\cG$ and an infinite
word $\xi\in\dT$ with infinitely many zeros, one has $g|_w=\1$ for a
sufficiently long beginning $w$ of $\xi$.  This claim implies the lemma.
Indeed, in the case $g(\xi)=\xi$ the action of $g$ fixes all infinite words
with beginning $w$, which form an open neighborhood of $\xi$.

Let $R$ be the set of all $g\in\cG$ such that the claim holds true for $g$
and any $\xi\in\dT$ with infinitely many zeros.  The set $R$ contains the
generating set $S$.  Indeed, $a|_w=\1$ for any nonempty word $w\in X^*$ and
$b|_w=c|_w=d|_w=\1$ for any word $w$ that contains a zero which is not the
last letter of $w$.  Now suppose $g,h\in R$ and consider an arbitrary
$\xi\in\dT$ with infinitely many zeros.  Then $h|_w=\1$ for a sufficiently
long beginning $w$ of $\xi$.  Lemma \ref{grig2} implies that the infinite
word $h(\xi)$ also has infinitely many zeros.  Since $h(w)$ is a beginning
of $h(\xi)$ and $g\in R$, we have $g|_{h(w)}=\1$ provided $w$ is long
enough.  Since $(gh)|_w=g|_{h(w)}h|_w$, we have $(gh)|_w=\1$ provided $w$
is long enough.  Thus $gh\in R$.  That is, the set $R$ is closed under
multiplication.  Since $S\subset R$ and all generators are involutions, it
follows that $R=\cG$.
\end{proof}

The infinite word $\xi_0=111\dots$ (also denoted $1^\infty$) is an
exceptional point for the action $\be$.

\begin{lemma}\label{grig6}
The quotient of $\St_\be(\xi_0)$ by $\St^o_\be(\xi_0)$ is the Klein
$4$-group.  The coset representatives are $\1,b,c,d$.
\end{lemma}

\begin{proof}
Recall that $H=\{\1,b,c,d\}$ is a subgroup of $\cG$ isomorphic to the Klein
$4$-group.  Clearly, $H\subset\St_\be(\xi_0)$.  We are going to show that
$H\cap\St^o_\be(\xi_0)=\{\1\}$ and $\St_\be(\xi_0)=\St^o_\be(\xi_0)H$,
which implies the lemma.

For any positive integer $n$ let $\eta_n$ denote the infinite word over the
alphabet $X$ that has a single zero in the position $n$.  The sequence
$\eta_1,\eta_2,\dots$ converges to $\xi_0$.  One observes that any of the
generators $b$, $c$, and $d$ fixes $\eta_n$ only if $n$ leaves a specific
remainder under division by $3$ ($0$ for $b$, $2$ for $c$, and $1$ for
$d$).  It follows that $H\cap\St^o_\be(\xi_0)=\{\1\}$.

Now let us show that any $g\in\St_\be(\xi_0)$ is contained in the set
$\St^o_\be(\xi_0)H$.  The proof is by strong induction on the length $n$ of
$g$, which is the smallest possible number of factors in an expansion
$g=s_m\dots s_2s_1$ such that each $s_i\in S$.  The case $n=0$ is trivial
as $\1$ is the only element of length $0$.  Assume that the claim is true
for all elements of length less than some $n>0$ and consider an arbitrary
element $g\in\St_\be(\xi_0)$ of length $n$.  We have $g=s_n\dots s_2s_1$,
where each $s_i$ is a generator from $S$.  Let
$\xi_k=(s_k\dots s_2s_1)(\xi_0)$, $k=1,2,\dots,n$.  If $\xi_k=\xi_0$ for
some $0<k<n$, then $g_1=s_n\dots s_{k+1}$ and $g_2=s_k\dots s_2s_1$ both
fix $\xi_0$.  Since the length of $g_1$ and $g_2$ is less than $n$, they
belong to $\St^o_\be(\xi_0)H$ by the inductive assumption.  As
$\St^o_\be(\xi_0)H$ is a group, so does $g=g_1g_2$.  If $\xi_k\ne\xi_0$ for
all $0<k<n$, then $s_{i+1}|_{w_i}=\1$ for any $0\le i<n$ and sufficiently
long beginning $w_i$ of the infinite word $\xi_i$.  It follows that
$g|_w=\1$ for a sufficiently long beginning $w$ of $\xi_0$.  Thus
$g\in\St^o_\be(\xi_0)$.
\end{proof}

Recall that we consider the Schreier graphs of the group $\cG$ relative to
the generating set $S=\{a,b,c,d\}$ as graphs with undirected edges.  The
Schreier graphs of all orbits of the action $\be$ except $O_\be(\xi_0)$ are
similar.  Any vertex is joined to two other vertices.  Moreover, it is
joined to one of the neighbors by a single edge labeled $a$ and to the
other neighbor by two edges.  Also, there is one loop at each vertex.
Hence the Schreier graph has a linear structure (see Figure \ref{fig1}) and
all such graphs are isomorphic as graphs with unlabeled edges.  The
Schreier graph of the orbit of $\xi_0=1^\infty$ is different in that there
are three loops labeled $b$, $c$, and $d$ at the vertex $\xi_0$ (see Figure
\ref{fig2}).

Let $F:\dT\to\Sch(\cG,S)$ be the mapping that assigns to any point on the
boundary of the binary rooted tree $\cT$ its marked Schreier graph under
the action $\be$.  Using notation of Section \ref{sch},
$F(\xi)=\Ga_{\Sch}^*(\cG,S;\be,\xi)$ for all $\xi\in\dT$.

\begin{lemma}\label{grig7}
The graph $F(\xi_0)$ is an isolated point in the image $F(\dT)$.
\end{lemma}

\begin{proof}
Let $\Ga_0$ denote the marked graph with a single vertex and three loops
labeled $b$, $c$, and $d$.  Recall that $\cU(\Ga_0,\emptyset)$ is an open
subset of $\MG_0$ consisting of all graphs in $\MG_0$ that have a subgraph
isomorphic to $\Ga_0$.  Hence $\cU(\Ga_0,\emptyset)\cap\Sch(\cG,S)$ is an
open subset of $\Sch(\cG,S)$.  Given $\xi\in\dT$, the graph $F(\xi)$
belongs to that open subset if and only if $a(\xi)\ne\xi$ and
$b(\xi)=c(\xi)=d(\xi)=\xi$.  The latter conditions are satisfied only for
$\xi=\xi_0$.  The lemma follows.
\end{proof}

\begin{figure}[t]
\centerline{\includegraphics{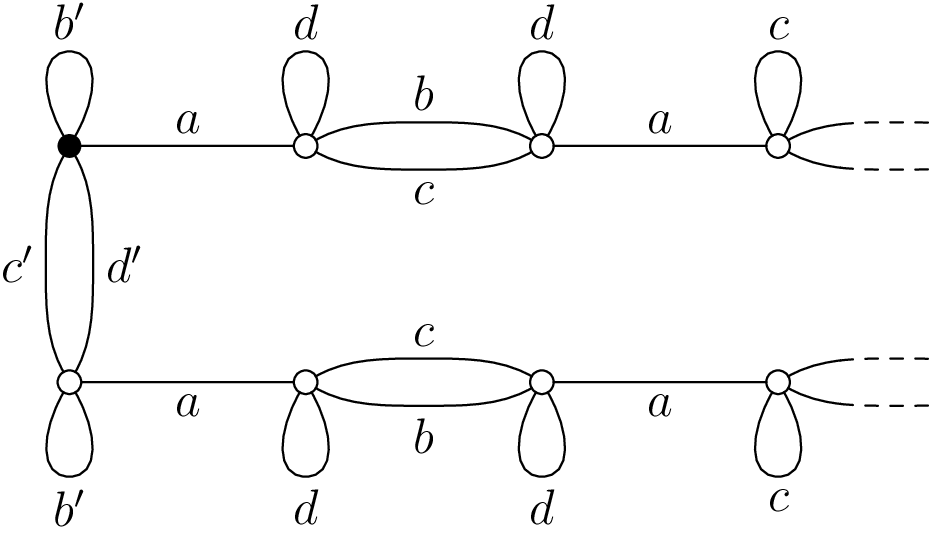}}
\caption{Limit graphs $\De^*_0,\De^*_1,\De^*_2$.}
\label{fig3}
\end{figure}

It turns out that the image $F(\dT)$ is not closed in $\Sch(\cG,S)$.  The
following construction will help to describe the closure of $F(\dT)$.  Let
us take two copies of the Schreier graph $\Ga_{\Sch}(\cG,S;\be,\xi_0)$.  We
remove two out of three loops at the vertex $\xi_0$ (loops with the same
labels in both copies) and replace them with two edges joining the two
copies.  Let $c'$ and $d'$ denote labels of the removed loops and $b'$
denote the label of the retained loop.  Then $b',c',d'$ is a permutation
of $b,c,d$.  To be rigorous, the new graph has the vertex set
$O_\be(\xi_0)\times\{0,1\}$, the set of edges $O_\be(\xi_0)\times\{0,1\}
\times S$, and the set of labels $S$.  An arbitrary edge $(\xi,i,s)$ has
beginning $(\xi,i)$ and label $s$.  The end of this edge is $(s(\xi),i)$
unless $\xi=\xi_0$ and $s=c'$ or $s=d'$, in which case the end is
$(s(\xi),1-i)=(\xi_0,1-i)$.  There are three ways to perform the above
construction depending on the choice of $b'$.  We denote by $\De_0$,
$\De_1$, and $\De_2$ the graphs obtained when $b'=b$, $b'=d$, and $b'=c$,
respectively.  Further, for any $i\in\{0,1,2\}$ we denote by $\De^*_i$ a
marked graph obtained from $\De_i$ by marking the vertex $(\xi_0,0)$ (see
Figure \ref{fig3}).  

Consider an arbitrary sequence of points $\eta_1,\eta_2,\dots$ in $\dT$
such that $\eta_n\to\xi_0$ as $n\to\infty$, but $\eta_n\ne\xi_0$.  Let
$z_n$ denote the position of the first zero in the infinite word $\eta_n$.

\begin{lemma}\label{grig8}
The marked Schreier graphs $F(\eta_n)$ converge to $\De^*_i$, $0\le i\le
2$, as $n\to\infty$ if $z_n\equiv i\bmod3$ for large $n$.
\end{lemma}

\begin{proof}
For any $n\ge1$ we define a map $f_n:O_\be(\xi_0)\times\{0,1\}\to
O_\be(\eta_n)$ as follows.  Given $\xi\in O_\be(\xi_0)$ and $x\in\{0,1\}$,
let $f_n(\xi,x)$ be an infinite word obtained from $\eta_n$ after replacing
the first $z_n-1$ letters by the first $z_n-1$ letters of $\xi$ and adding
$x\bmod2$ to the $(z_n+1)$-th letter.  Clearly, $f_n(\xi_0,0)=\eta_n$.  Let
$i_n$ be the remainder of $z_n$ under division by $3$.  One can check that
the restriction of $f_n$ to the vertex set of the closed ball
$\oB_{\De^*_{i_n}}((\xi_0,0),N)$ is an isomorphism of this ball with the
closed ball $\oB_{F(\eta_n)}(\eta_n,N)$ whenever $N\le 2^{z_n-2}$.
Therefore $\de(F(\eta_n),\De^*_{i_n})\to0$ as $n\to\infty$.
\end{proof}

One consequence of Lemma \ref{grig8} is that the graphs $\De_0$, $\De_1$,
and $\De_2$ are Schreier graphs of the group $\cG$.  By construction, each
of these graphs admits a nontrivial automorphism, which interchanges
vertices corresponding to the same vertex of $\Ga_{\Sch}(\cG,S;\be,\xi_0)$.
This property distinguishes $\De_0$, $\De_1$, and $\De_2$ from the Schreier
graphs of orbits of the action $\be$.

\begin{lemma}\label{grig9}
The Schreier graphs $\Ga_{\Sch}(\cG,S;\be,\xi)$, $\xi\in\dT$ do not admit
nontrivial automorphisms.  The graphs $\De_0$, $\De_1$, and $\De_2$ admit
only one nontrivial automorphism.
\end{lemma}

\begin{proof}
It follows from Proposition \ref{sch4} and Lemma \ref{grig4} that marked
Schreier graphs $F(\xi)$ and $F(\eta)$ are isomorphic only if $\xi=\eta$.
Therefore the Schreier graphs $\Ga_{\Sch}(\cG,S;\be,\xi)$, $\xi\in\dT$
admit no nontrivial automorphisms.

The graphs $\De_0$, $\De_1$, and $\De_2$ have linear structure.  Namely,
one can label their vertices by $v_j$, $j\in\bZ$ so that each $v_j$ is
adjacent only to $v_{j-1}$ and $v_{j+1}$.  If $f$ is an automorphism of
such a graph, then either $f(v_j)=v_{n-j}$ for some $n\in\bZ$ and all
$j\in\bZ$ or $f(v_j)=v_{n+j}$ for some $n\in\bZ$ and all $j\in\bZ$.  Assume
that some $\De_i$ has more than one nontrivial automorphism.  Then we can
choose $f$ above so that the latter option holds with $n\ne0$.  Take any
path in $\De_i$ that begins at $v_0$ and ends at $v_n$ and let $w$ be the
code word of that path.  Since $f^m(v_0)=v_{mn}$ and $f^m(v_n)=v_{(m+1)n}$
for any integer $m$, the path in $\De_i$ with beginning $v_{mn}$ and code
word $w$ ends at $v_{(m+1)n}$.  It follows that for any integer $m>0$ the
path with beginning $v_0$ and code word $w^m$ ends at $v_{mn}$.  In
particular, this path is not closed.  However every element of the
Grigorchuk group $\cG$ is of finite order (see \cite{G1}) so that for some
$m>0$ the reversed word $w^m$ equals $\1$ when regarded as a product in
$\cG$.  This conradicts with Proposition \ref{sch1}.  Thus the graph
$\De_i$ admits only one nontrivial automorphism.
\end{proof}

\begin{figure}[t]
\centerline{\includegraphics{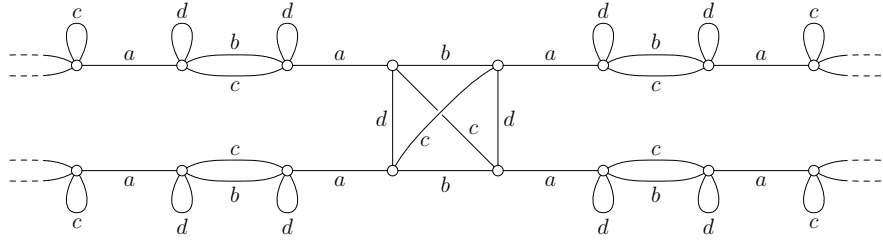}}
\caption{The Schreier coset graph of $\St^o_\be(\xi_0)$.}
\label{fig4}
\end{figure}

\begin{lemma}\label{grig10}
The Schreier graph $\Ga_{\Sch}(\cG,S;\be,\xi_0)$ is a double quotient of
each of the graphs $\De_0$, $\De_1$, and $\De_2$.  On the other hand, each
of the graphs $\De_0$, $\De_1$, and $\De_2$ is a double quotient of the
Schreier coset graph $\Ga_{\coset}(\cG,S;\St^o_\be(\xi_0))$.
\end{lemma}

\begin{proof}
The Schreier coset graph of the subgroup $\St^o_\be(\xi_0)$ is shown in
Figure \ref{fig4}.  In view of Lemmas \ref{grig6} and \ref{grig9}, the
automorphism group of this graph is the Klein $4$-group.  The quotient of
the graph by the entire automorphism group is the Schreier graph of the
orbit of $\xi_0$.  The quotients by subgroups of order $2$ are the graphs
$\De_0$, $\De_1$, and $\De_2$.
\end{proof}

Now it remains to collect all parts in Theorems \ref{main1} and
\ref{main2}.

\begin{proofof}{Theorem \ref{main1}}
We are concerned with the mapping $F:\dT\to\Sch(\cG,S)$ given by $F(\xi)
=\Ga_{\Sch}^*(\cG,S;\be,\xi)$.  Let us also consider a mapping
$\psi:\dT\to\Sub(\cG)$ given by $\psi(\xi)=\St_\be(\xi)$ and a mapping
$f:\Sub(\cG)\to\Sch(\cG,S)$ given by $f(H)=\Ga^*_{\coset}(\cG,S;H)$.  By
Proposition \ref{sch4}, $F(\xi)=f(\psi(\xi))$ for all $\xi\in\dT$.  By
Proposition \ref{sub5}, $f$ is a homeomorphism.  Lemma \ref{grig4} implies
that the mapping $\psi$ is injective.  It is Borel measurable due to Lemma
\ref{sub4}.  Also, $\psi$ is continuous at a point $\xi\in\dT$ if and only
if $\St^o_\be(\xi)=\St_\be(\xi)$.  Lemmas \ref{grig5} and \ref{grig6} imply
that the latter condition fails only if the infinite word $\xi$ contains
only finitely many zeros.  According to Lemma \ref{grig2}, an equivalent
condition is that $\xi$ is in the orbit of $\xi_0=1^\infty$ under the
action $\be$.  Since the mapping $F$ is $f$ postcomposed with a
homeomorphism, it is also injective, Borel measurable, and continuous
everywhere except the orbit of $\xi_0$.

By Lemma \ref{grig7}, the graph $F(\xi_0)$ is an isolated point of the
image $F(\dT)$.  Since $F(g(\xi))=\cA_g(F(\xi))$ for any $\xi\in\dT$ and
$g\in\cG$ and since the action $\cA$ is continuous (see Proposition
\ref{sch2}), the graph $F(g(\xi_0))$ is an isolated point of $F(\dT)$ for
all $g\in\cG$.  On the other hand, if $\xi\in\dT$ is not in the orbit of
$\xi_0$, then the graph $F(\xi)$ is not an isolated point of $F(\dT)$ as
the mapping $F$ is injective and continuous at $\xi$.

It follows from Lemma \ref{grig9} that the image $F(\dT)$ and the orbits
$O_{\cA}(\De^*_i)$, $i\in\{0,1,2\}$ are disjoint sets.  Note that the orbit
$O_{\cA}(\De^*_i)$ consists of marked graphs obtained from the graph
$\De_i$ by marking an arbitrary vertex.  Lemma \ref{grig8} implies the
union of those $4$ sets is the closure of $F(\dT)$.

Finally, the statement (v) of Theorem \ref{main1} follows from Lemma
\ref{grig10}.
\end{proofof}

\begin{proofof}{Theorem \ref{main2}}
Lemma \ref{tree5} combined with Lemma \ref{grig4} implies that the action
of $\cG$ on the closure of $F(\dT)$ is a continuous extension of the action
$\be$.  The extension is one-to-one everywhere except for the orbit
$O_\be(\xi_0)$ where it is four-to-one.  Namely, for any $g\in\cG$ the
point $g(\xi_0)$ is covered by $4$ graphs $F(g(\xi_0))$, $\cA_g(\De^*_0)$,
$\cA_g(\De^*_1)$, and $\cA_g(\De^*_0)$.  According to Theorem \ref{main1},
the graph $F(g(\xi_0))$ is an isolated point of the closure of $F(\dT)$.
When we restrict our attention to the set $\Om$ of non-isolated points of
the closure, we still have a continuous extension of the action $\be$, but
it is three-to-one on the orbit $O_\be(\xi_0)$.

By Lemma \ref{grig2}, the group $\cG$ acts transitively on each level of
the binary rooted tree $\cT$.  Then Proposition \ref{tree1} implies that
the action $\be$ is minimal and uniquely ergodic, the only invariant Borel
probability measure being the uniform measure on $\dT$.
Since the action of $\cG$ on the set $\Om$ is a continuous extension of the
action $\be$ that is one-to-one except for a countable set and since this
action has no finite orbits, it follows that the action is minimal,
uniquely ergodic, and isomorphic to $\be$ as the action with an invariant
measure.
\end{proofof}

\bigskip

\begin{raggedright}
\sc
Department of Mathematics\\
Mailstop 3368\\
Texas A\&M University\\
College Station, TX 77843-3368\\[3mm]
\it E-mail: \ \tt yvorobet@math.tamu.edu
\end{raggedright}


\begin{thebibliography}{G}

\bibitem{G0}
R. Grigorchuk, \textit{On Burnside's problem on periodic groups}.
Funct. Anal. Appl. {\bf 14} (1980), 41--43.

\bibitem{G1}
R.~Grigorchuk, \textit{Solved and unsolved problems around one group}.
L.~Bartholdi (ed.) et al., Infinite groups: geometric, combinatorial and
dynamical aspects.  Basel, Birkh\"auser. {\sl Progress in Mathematics\/}
{\bf 248}, 117--218 (2005).

\bibitem{G2}
R.~Grigorchuk, \textit{Some topics in the dynamics of group actions on
rooted trees}.
Proc. Steklov Inst. Math. {\bf 273} (2011), no. 1, 64--175.

\bibitem{K}
A. S. Kechris, \textit{Classical descriptive set theory}.  Berlin,
Springer.  {\sl Graduate Texts in Mathematics\/} 156 (1995).

\bibitem{V}
A.~Vershik, \textit{Nonfree actions of countable groups and their
characters}.
J. Math. Sci., NY {\bf 174} (2011), no. 1, 1--6.

\end{thebibliography}
\end{document}